\documentclass[fullpage, 11pt]{article}

\usepackage[english]{babel}
\usepackage{amsfonts}
\usepackage{amssymb}
\usepackage{amsmath}
\usepackage{latexsym}
\usepackage{amsthm}
\usepackage{epsfig}
\usepackage{subfig}
\usepackage{color}

\newcommand{\bb}{\mathbb}
\newcommand{\conv}{\mathrm{conv}}
\newcommand{\cone}{\mathrm{cone}}
\newcommand{\lin}{\mathrm{lin}}
\newcommand{\R}{\bb R}
\newcommand{\Q}{\bb Q}
\newcommand{\Z}{\bb Z}

\newcommand{\floor}[1]{\lfloor#1\rfloor}
\newcommand{\ceil}[1]{\lceil#1\rceil}
\newcommand{\proj}{\mathrm{proj}}
\newcommand{\intr}{\mathrm{\bf int}}
\newcommand{\relint}{\mathrm{\bf relint}}
\newcommand{\cl}{\mathrm{\bf cl}}
\newcommand{\bd}{\mathrm{\bf bd}}
\newcommand{\ol}{\overline}
\newcommand{\rec}{\mathrm{rec}}

\newcommand{\V}{\mathcal V}
\newcommand{\W}{\mathcal W  }

\newcommand{\old}[1]{{}}

\newtheorem{prop}{Proposition}
\newtheorem{theorem}[prop]{Theorem}
\newtheorem{corollary}[prop]{Corollary}
\newtheorem{lemma}[prop]{Lemma}

\newtheorem{claim}{Claim}

\newtheorem{definition}[prop]{Definition}
\newtheorem{remark}[prop]{Remark}

\newenvironment{pf}{\begin{trivlist} \item[] {\em Proof:}}{\hfill $
\Box$
                       \end{trivlist}}
\newcommand{\sm}{\setminus}

\newcommand{\aff}{\mathrm{aff}}

\def\st{\,|\,}

\old{ \addtolength{\oddsidemargin}{-30pt}
\addtolength{\evensidemargin}{-30pt} \addtolength{\textwidth}{60pt}

\addtolength{\topmargin}{-22pt} \addtolength{\textheight}{32pt} }

\begin{document}
\title{Maximal lattice-free convex sets in linear subspaces}

\author{%
Amitabh Basu \\
Carnegie Mellon University,
abasu1@andrew.cmu.edu\\\\
Michele Conforti \\
Universit\`a di Padova,
conforti@math.unipd.it \\ \\
G\'erard Cornu\'ejols \thanks{
Supported by   NSF grant CMMI0653419,
ONR grant N00014-03-1-0188 and ANR grant BLAN06-1-138894.} \\
Carnegie Mellon University and Universit\'e d'Aix-Marseille \\
gc0v@andrew.cmu.edu \\ \\
Giacomo Zambelli \\
Universit\`a di Padova,
giacomo@math.unipd.it
}

\date{March 30, 2009, Revised April 21, 2010.}

\maketitle

\begin{abstract}
We consider a model that arises in integer programming, and show that all irredundant inequalities are obtained from maximal lattice-free convex sets in an affine subspace. We also show that these sets are polyhedra. The latter result extends a theorem of Lov\'asz characterizing maximal lattice-free convex sets in $\R^n$.
\end{abstract}
%\normalsize

\section{Introduction}
The study of maximal lattice-free convex sets dates back to
Minkowski's work on the geometry of numbers. Connections between
integer programming and the geometry of numbers were investigated in
the 1980s starting with the work of Lenstra~\cite{lenstra}. See
Lov\'asz~\cite{Lovasz} for a survey. Recent work in cutting plane
theory~\cite{AndLouWeiWol},\cite{AndLouWei},\cite{ALW1},\cite{ALW2},\cite{AWW},\cite{BBCM},\cite{BorCor},\cite{cm},\cite{Dey-Wol},\cite{DW},\cite{Esp},\cite{go07},\cite{giacomo}
has generated renewed interest in the study of maximal lattice-free
convex sets. In this paper we further pursue this line of research.
In the first part of the paper we consider convex sets in an affine
subspace of $\R^n$ that are maximal with the property of not
containing integral points in their relative interior. When this
affine subspace is rational, these convex sets are characterized by
a result of Lov\'asz~\cite{Lovasz}. The extension to irrational
subspaces appears to be new, and it has already found an application
in the proof of a key result in~\cite{BorCor}. It is also used to
prove the main result in the second part of this paper: We consider
a model that arises in integer programming, and show that all
irredundant inequalities are obtained from maximal lattice-free
convex sets in an affine subspace.
\bigskip

Let $W$ be an affine subspace of $\R^n$. Assume that $W$ contains
an integral point, i.e. $W\cap \Z^n\neq
\emptyset$. We say that a set $B\subset \R^n$ is a {\em maximal
lattice-free convex set in $W$} if $B\subset W$, $B$ is convex, $B$
has no integral point in its interior with respect to the topology
induced on $W$ by $\R^n$, and $B$ is inclusionwise maximal with
these three properties. This definition implies that either $B$
contains no integral point in its relative interior or $B$ has
dimension strictly less than $W$.

The subspace $W$ is said to be {\em rational} if it is generated by
the integral points in $W$. So, if we denote by $V$ the affine hull
of the integral points in $W$, $V=W$ if and only if $W$ is rational.
If $W$ is not rational, then the inclusion $V\subset W$ is strict.
When $W$ is not rational, we will also say that $W$ is {\em
irrational}. An example of an irrational affine subspace $W\subseteq
\R^3$ is the set of points satisfying the equation $x_1 + x_2 +
\sqrt{2}x_3 = 1$. The affine hull $V$ of $W \cap \Z^3$ is the set of
points satisfying the equations $x_1 + x_2 = 1, \; x_3 = 0$.

\begin{theorem}\label{thm:main-intr} Let $W \subset \R^n$ be an affine space containing an integral point and $V$ the affine hull of $W \cap \mathbb{Z}^n$. A set $S\subset W$ is a maximal lattice-free convex set of $W$ if and only if one of the following holds:
\begin{itemize}
\item[(i)] $S$ is a polyhedron in $W$ whose dimension equals $\dim(W)$,
$S\cap V$ is a maximal lattice-free convex set of $V$ whose
dimension equals $\dim(V)$, the facets of $S$ and $S\cap V$ are in
one-to-one correspondence and for every facet $F$ of $S$, $F\cap V$
is the facet of $S\cap V$ corresponding to $F$;
\item[(ii)] $S$ is an hyperplane of $W$ of the form $v+L$, where $v\in S$ and $L\cap V$ is an irrational hyperplane of $V$;
\item[(iii)]  $S$ is a half-space of $W$ that contains $V$ on its boundary.
\end{itemize}
\end{theorem}

\begin{figure}[htbp]\begin{center}
\includegraphics[scale=0.65]{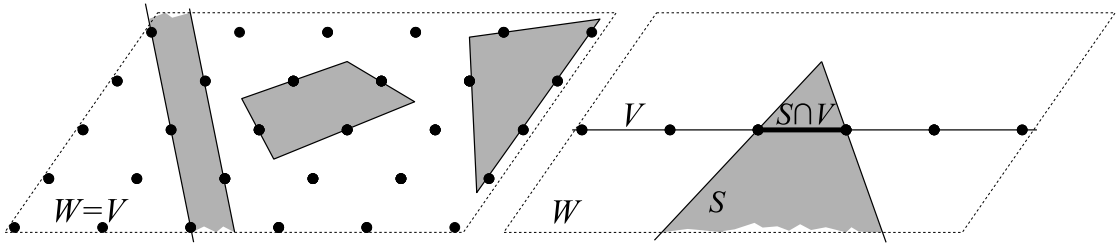}
\caption{\label{fig:2D}
Maximal lattice-free convex sets in a 2-dimensional subspace (Theorem~\ref{thm:main-intr}(i)).}
\end{center}
\end{figure}

A characterization of maximal lattice-free convex sets of $V$, needed in $(i)$ of the previous theorem, is given by the following.

\begin{theorem}\label{thm:lattice-free-intr} {\em (Lov\'asz \cite{Lovasz})} Let $V$ be a rational affine subspace
of $\R^n$ containing an integral point. A set $S\subset V$ is a maximal lattice-free convex set of $V$ if and only if one of the following holds:
\begin{itemize}
\item[(i)] $S$ is a polyhedron of the form $S= P+L$ where $P$ is a polytope,
$L$ is a rational linear space, $\dim(S)=\dim(P)+\dim(L)=\dim(V)$, $S$ does not
contain any integral point in its relative interior and there is an integral point in the relative interior of each facet of $S$;
\item[(ii)] $S$ is an affine hyperplane of $V$ of the form $v+L$, where $v\in S$ and $L$ is an irrational hyperplane of $V$;
\end{itemize}
\end{theorem}

The polyhedron $S= P+L$ in Theorem~\ref{thm:lattice-free-intr}(i) is
called a {\em cylinder over the polytope} $P$ and can be shown to
have at most $2^{\dim(P)}$ facets \cite{doignon}.

Theorem~\ref{thm:main-intr} is new and it is used in the proof of
our main result about integer programming,
Theorem~\ref{thm:min-ineq-intr} below. It is also used to prove the
last theorem in~\cite{BorCor}. Theorem~\ref{thm:lattice-free-intr}
is due to Lov\'asz (\cite{Lovasz} Proposition 3.1). Lov\'asz only
gives a sketch of the proof and it is not clear how case (ii) in
Theorem~\ref{thm:lattice-free-intr} arises in his sketch or in the
statement of his proposition. Therefore in
Section~\ref{sec:max-conv} we will prove both theorems.

Figure~\ref{fig:2D} shows examples of maximal lattice-free convex
sets in a 2-dimensional affine subspace $W$ of $\R^3$. We denote by
$V$ the affine space generated by $W\cap \Z^3$. In the first picture
$W$ is rational, so $V=W$, while in the second one $V$ is a subspace
of $W$ of dimension 1.\medskip

We now give an example of Theorem~\ref{thm:main-intr}(ii). Let
$W\subseteq \R^4$ be the set of points satisfying the equation $x_1
+ x_2 + x_3 + \sqrt{2}x_4 = 1$. The affine hull $V$ of $W \cap \Z^4$
is the set of points satisfying the equations $x_1 + x_2 + x_3 = 1,
\; x_4 = 0$. The set $S \subset W$ defined by the equations $x_1+x_2+x_3+\sqrt{2}x_4 =1$,
 $x_1 +\sqrt{2}x_2 = 1$ satisfies
Theorem~\ref{thm:main-intr}(ii). Indeed, $\dim(W) = \dim(S) + 1 =
3$. Furthermore, $\dim(V) = 2$ and $S\cap V$ is the line satisfying
the equations $x_1 + x_2 + x_3 = 1,\; x_1 + \sqrt{2}x_2 = 1,\; x_4 =
0$ and it is an irrational subspace since the only integral point it
contains is $(1,0,0,0)$.
\bigskip

Next we highlight the relation between lattice-free convex sets and
valid inequalities in integer programming. This was first observed
by Balas~\cite{bal1}.

Suppose we consider $q$ rows of the optimal tableau of the LP
relaxation of a given MILP, relative to $q$ basic integer variables
$x_1,\ldots, x_q$. Let $s_1,\ldots, s_k$ be the nonbasic variables,
and $f\in\R^q$ be the vector of components of the optimal basic
feasible solution. The tableau restricted to these $q$ rows is of
the form
$$x=f+\sum_{j=1}^k r^js_j,\quad x\geq 0\mbox{ integral},\, s\geq 0,\, \mbox { and } s_j\in \Z, j\in I,$$
where $r^j\in\R^q$, $j=1,\ldots, k$, and $I$ denotes the set of integer nonbasic variables. Gomory \cite{gom69} proposed to consider the relaxation of the above problem obtained by dropping the nonnegativity conditions $x\geq 0$. This gives rise to the so called {\em corner polyhedron}. A further relaxation is obtained by also dropping the integrality conditions on the nonbasic variables, obtaining the mixed-integer set
$$x=f+\sum_{j=1}^k r^js_j,\; x\in\Z^q,\, s\geq 0.$$
Note that, since $x\in\R^q$ is completely determined by $s\in\R^k$, the above is equivalent to
\begin{equation}\label{eq:Rf-finite}f+\sum_{j=1}^k r^js_j\in\Z^q,\quad s\geq 0.\end{equation}
We denote by $R_f(r^1,\ldots,r^k)$ the set of points $s$ satisfying
(\ref{eq:Rf-finite}). The above relaxation was studied by~Andersen
et al.~\cite{AndLouWeiWol} in the case of two rows and Borozan and
Cornu\'ejols~\cite{BorCor} for the general case. In these papers
they showed that the irredundant valid inequalities
for~$R_f(r^1,\ldots,r^k)$ correspond to maximal lattice-free convex
sets in $\R^q$. In~\cite{AndLouWeiWol,BorCor} data are assumed to be
rational. Here we consider the case were $f, r^1,\ldots,r^k$ may
have irrational entries.

Let $W=\langle r^1,\ldots,r^k\rangle$ be the linear space generated
by $r^1,\ldots, r^k$. Note that, for every $s\in
R_f(r^1,\ldots,r^k)$, the point
$f+\sum_{j=1}^kr^js_j\in(f+W)\cap\Z^q$, hence we assume $f+W$
contains an integral point. Let $V$ be the affine hull of $(f+W)\cap
\Z^q$. Notice that $f+W$ and $V$ coincide if and only if $W$ is a
rational space. Borozan and Cornu\'ejols \cite{BorCor} proposed to
study the following {\em semi-infinite relaxation}, which is a
special case of Gomory and Johnson's group problem~\cite{gj}. Let
$R_f(W)$ be the set of points $s=(s_r)_{r\in W}$ of $\R^{W}$
satisfying
\begin{eqnarray}\label{eq:Rf}
&&f+\sum_{r\in W} r s_r\in \Z^q\nonumber\\
&&s_r\geq 0,\quad r\in W\\
&&s \in \mathcal W\nonumber
\end{eqnarray}
where $\mathcal W$ is the set of all $s\in\R^{W}$ with
{\em finite support}, i.e. the set $\{r\in W\st s_r>0\}$ has finite cardinality. Notice that $R_f(r^1,\ldots,r^k)=R_f(W)\cap\{s\in\W\st s_r=0\mbox{ for all } r\neq r^1,\ldots,r^k\}$.\medskip

Given a function $\psi\,:\,W\rightarrow \R$ and $\alpha\in\R$, the linear inequality
\begin{equation}\label{eq:valid}\sum_{r\in W}\psi(r)s_r\geq\alpha\end{equation}
is {\em valid for $R_f(W)$} if it is satisfied by every $s\in R_f(W)$.

Note that, given a valid inequality (\ref{eq:valid}) for $R_f(W)$, the inequality
$$\sum_{j=1}^k \psi(r^j)s_j\geq \alpha$$ is valid for $R_f(r^1,\ldots,r^k)$.
Hence a characterization of valid linear inequalities for $R_f(W)$
provides a characterization of valid linear inequalities for
$R_f(r^1,\ldots,r^k)$.
\medskip

Next we observe how maximal lattice-free convex sets in $f+W$ give
valid linear inequalities for $R_f(W)$. Let $B$ be a maximal
lattice-free convex set in $f+W$ containing $f$ in its interior.
Since, by Theorem~\ref{thm:main-intr}, $B$ is a polyhedron and since
$f$ is in its interior, there exist $a_1,\ldots,a_t\in\R^q$ such
that $B=\{x\in f+W\st a_i(x-f)\leq 1,\, i=1\ldots, t\}$. We define
the function $\psi_B\,:\, W\rightarrow \R$ by
\begin{equation}
\label{eq:psi_B}\psi_B(r)=\max_{i=1,\ldots,t}a_ir.
\end{equation} Note that the function $\psi_B$ is {\em subadditive}, i.e.
$\psi_B(r)+\psi_B(r')\geq \psi_B(r+r')$, and {\em positively
homogeneous}, i.e. $\psi_B(\lambda r)=\lambda\psi_B(r)$ for every
$\lambda\geq 0$. We claim that $$\sum_{r\in W}\psi_B(r)s_r\geq 1$$
is valid for $R_f(W)$.

Indeed, let $s\in R_f(W)$, and $x=f+\sum_{r\in W} rs_r$. Since
$x\in\Z^n$ and $B$ is lattice-free, $x\notin\intr(B)$. Then
$$\sum_{r\in W}\psi_B(r)s_r=\sum_{r\in
W}\psi_B(rs_r)\geq\psi_B(\sum_{r\in W}rs_r)=\psi_B(x-f)\geq 1,$$
where the first equation follows from positive homogeneity, the
first inequality follows from subadditivity of $\psi_B$ and the last
one follows from the fact that $x\notin\intr(B)$.

We will show that all nontrivial irredundant valid linear
inequalities for $R_f(W)$ are indeed of the type described above.
Furthermore, if $W$ is irrational, we will see that $R_f(W)$ is
contained in a proper affine subspace of $\W$, so each inequality
has infinitely many equivalent forms. Note that, by definition of
$\psi_B$, $\psi_B(r)>0$ if $r$ is not in the recession cone of $B$,
$\psi_B(r) < 0$ when $r$ is in the interior of the recession cone of
$B$, while $\psi_B(r)=0$ when $r$ is on the boundary of the
recession cone of $B$. We will show that one can always choose a
form of the inequality so that $\psi_B$ is a nonnegative function.
We make this more precise in the next theorem.
\bigskip

Given a point $s\in R_f(W)$, then $f+\sum_{r\in W} r s_r\in \Z^q\cap
(f+W)$. Recall that we denote by $V$ the affine hull of $\Z^q\cap
(f+W)$. Thus $R_f(W)$ is contained in the affine subspace $\mathcal
V$ of $\mathcal W$ defined as $$\mathcal V=\{s\in\mathcal W\st
f+\sum_{r\in W} r s_r\in V\}.$$ Observe that, given
$C\in\R^{\ell\times q}$ and $d\in\R^\ell$ such that $V=\{x\in f+W\st
Cx=d\}$, we have \begin{equation}\label{eq:V}\V=\{s\in\W\st
\sum_{r\in W}(Cr)s_r=d-Cf\}.\end{equation}

Given two valid inequalities $\sum_{r\in W} \psi(r)s_r\geq \alpha$
and $\sum_{r\in W} \psi'(r)s_r\geq \alpha'$ for $R_f(W)$, we say
that they are {\em equivalent} if there exist $\rho>0$ and
$\lambda\in\R^\ell$ such that $\psi(r)=\rho\psi'(r)+\lambda^T Cr$
and $\alpha=\rho\alpha'+\lambda^T(d-Cf)$. Note that, if two valid
inequalities $\sum_{r\in W} \psi(r)s_r\geq \alpha$ and $\sum_{r\in
W} \psi'(r)s_r\geq \alpha'$ for $R_f(W)$ are equivalent, then
$\mathcal V \cap \{s \st \sum_{r\in W} \psi(r)s_r\geq \alpha\} =
\mathcal V \cap \{s \st \sum_{r\in W} \psi'(r)s_r\geq \alpha\}.$

A linear inequality $\sum_{r\in W}\psi(r)s_r\geq \alpha$ that is
satisfied by every element in $\{s\in \mathcal V\st s_r\geq 0\mbox{
for every }r\in W\}$ is said to be {\em trivial}.

We say that inequality $\sum_{r\in W}\psi(r)s_r\geq \alpha$ {\em
dominates} inequality $\sum_{r\in W}\psi'(r)s_r\geq \alpha$ if
$\psi(r) \leq \psi'(r)$ for all $r\in W$. Note that, for any $\bar
s\in\W$ such that $\bar s_r \geq 0$ for all $r \in W$, if $\bar s$
satisfies the first inequality, then $\bar s$ also satisfies the
second. A valid inequality $\sum_{r\in W}\psi(r)s_r\geq \alpha$ for
$R_f(W)$ is {\em minimal} if it is not dominated by any valid linear
inequality $\sum_{r\in W}\psi'(r)s_r\geq\alpha$ for $R_f(W)$ such
that $\psi'\neq\psi$. It is not obvious that nontrivial valid linear
inequalities are dominated by minimal ones. We will show that this
is the case. Note that it is not even obvious that minimal valid
linear inequalities exist.

We will show that, for any maximal lattice-free convex set $B$ of
$f+W$ with $f$ in its interior, the inequality $\sum_{r\in
W}\psi_B(r)s_r\geq 1$ is a minimal valid inequality for $R_f(W)$.
The main result is a converse, stated in the next
theorem. We need the notion of equivalent inequalities, which define
the same region in $\mathcal{V}$.

\begin{theorem}\label{thm:min-ineq-intr}
Every nontrivial valid linear inequality for $R_f(W)$ is dominated by a nontrivial minimal valid linear inequality for $R_f(W)$. \\
Every  nontrivial minimal valid linear inequality for $R_f(W)$ is equivalent to an inequality of the form $$\sum_{r \in W} \psi_B(r)s_r\geq 1$$ such that $\psi_B(r)\geq 0$ for all $r\in W$ and $B$ is a maximal lattice-free convex set in $f+W$ with $f$ in its interior.
\end{theorem}

This theorem generalizes earlier results of Borozan and
Cornu\'ejols~\cite{BorCor}.  In their setting it is immediate that
all valid linear inequalities are of the form $\sum_{r\in
W}\psi(r)s_r\geq 1$ with $\psi$ nonnegative. From this, it follows
easily that $\psi$ must be equal to $\psi_B$ for some maximal
lattice-free convex set $B$. The proof is much more complicated for
the case of $R_f(W)$ when $W$ is an irrational space. In this case,
valid linear inequalities might have negative coefficients. For
minimal inequalities, however, Theorem~\ref{thm:min-ineq-intr} shows
that there always exists an equivalent one where all coefficients
are nonnegative. The function $\psi_B$ is nonnegative if and only if
the recession cone of $B$ has empty interior. Although there are
nontrivial minimal valid linear inequalities arising from maximal
lattice-free convex sets whose recession cone is full dimensional,
Theorem~\ref{thm:min-ineq-intr} states that there always exists a
maximal lattice-free convex set whose recession cone is not full
dimensional that gives an equivalent inequality. A crucial
ingredient in showing this is a new result about sublinear functions
proved in~\cite{BaCoZa}.
\bigskip

In light of Theorem~\ref{thm:min-ineq-intr}, it is a natural question to
ask what is the subset of $\W$ obtained by intersecting the set of
nonnegative elements of $\V$ with all half-spaces defined by
inequalities $\sum_{r\in W} \psi(r)s_r\geq 1$ as in  Theorem~\ref{thm:min-ineq-intr}. In a
finite dimensional space, the intersection of all half-spaces
containing a given convex set $C$ is the closure of $C$. Things are
more complicated in infinite dimension. First of all, while in
finite dimension all norms are topologically equivalent, and thus
the concept of closure does not depend on the choice of a specific
norm, in infinite dimension different norms may produce different
topologies. Secondly, in finite dimensional spaces linear functions
are always continuous, while in infinite dimension there always
exist linear functions that are not continuous. In particular,
half-spaces (i.e. sets of points satisfying a linear inequality) are
not always closed in infinite dimensional spaces (see
Conway~\cite{Con} for example).

To illustrate this, note that if $\mathcal{W}$ is endowed
with the Euclidean norm, then $\mathbf{0}=(0)_{r\in W}$ belongs to
the closure of $\conv(R_f(W))$ with respect to this norm, as shown
next. Let $\bar x$ be an integral point in $f+W$
and let $\bar s$ be defined by
$$\bar s_r=\left\{\begin{array}{ll}
\frac 1 k & \mbox{ if } r=k(\bar x-f),\\
0 & \mbox{ otherwise}.
\end{array}
\right.
$$
Clearly, for every choice of $k$, $\bar s\in R_f(W)$, and for $k$ that goes to infinity the point $\bar s$ is
arbitrarily close to $\mathbf{0}$ with respect to the Euclidean

distance. Now, given a valid linear inequality $\sum_{r\in W} \psi(r)s_r\geq 1$
for $\conv(R_f(W))$, since $\sum_{r\in W} \psi(r)0=0$ the hyperplane
${\mathcal H}=\{s\in\W\,:\, \sum_{r\in W} \psi(r)s_r=1\}$ separates strictly
$\conv(R_f(W))$ from $\mathbf{0}$ even though $\mathbf{0}$ is in the closure of $\conv(R_f(W))$. This implies that ${\mathcal H}$ is not a closed
hyperplane of $\W$, and in particular the function $s\mapsto \sum_{r\in W} \psi(r)s_r$ is not
continuous with respect to the
Euclidean norm on $\W$.\medskip

A nice answer to our question is given by considering
a different norm on  $\W$. We endow $\W$ with
the norm $\|\cdot \|_H$ defined by
$$\|s\|_H=|s_0|+\sum_{r\in W\backslash\{0\}} \|r\| |s_r|.$$
It is straightforward to show that $\|\cdot\|_H$ is indeed a norm. Given $A\subset\W$, we denote by $\bar A$ the closure of
$A$ with respect to the norm $\|\cdot\|_H$.
\smallskip

Let $\mathcal{B}_W$ be the family of all maximal lattice-free
convex sets of $W$ with $f$ in their interior.

\begin{theorem}\label{thm:closure}
$$\ol\conv(R_f(W))=\left\{s\in\V\st \begin{array}{ll}\sum_{r\in W}\psi_B(r)s_r\ge 1 &B\in\mathcal{B}_W\\
s_r\geq 0\, & r\in W \end{array}\right\}.$$
\end{theorem}

Note that Theorems~\ref{thm:min-ineq-intr} and~\ref{thm:closure} are
new even when $W=\R^q$.
\medskip

A valid inequality $\sum_{r\in W}\psi(r)s_r\ge 1$ for $R_f(W)$ is
said to be {\em extreme} if there do not exist distinct functions
$\psi_1$ and $\psi_2$ satisfying $\psi \geq \frac 1 2 (\psi_1 +
\psi_2)$, such that $\sum_{r\in W}\psi_i(r)s_r\ge 1$, $i=1,2$, are
both valid for $R_f(W)$. The above definition is due to Gomory and
Johnson~\cite{gj}. Note that, if an inequality is not extreme, then
it is not necessary to define $\ol\conv(R_f(W))$.

The next theorem exhibits a correspondence between extreme
inequalities for the infinite model $R_f(\R^q)$ and extreme
inequalities for some finite problem $R_f(r^1, \ldots, r^k)$ where
$r^1, \ldots, r^k \in \R^q$. The theorem is very similar to a result
of Dey and Wolsey~\cite{DW}.

\begin{theorem}\label{thm:extreme} Let $B$ be a maximal lattice-free convex set in $\R^q$ with $f$ in
its interior. Let $L=\lin(B)$ and let $P=B\cap (f+L^\bot)$. Then
$B=P+L$, $L$ is a rational space, and $P$ is a polytope. Let
$v^1,\ldots,v^k$ be the vertices of $P$, and
$r^{k+1},\ldots,r^{k+h}$ be a rational basis of $L$. Define
$r^j=v^j-f$ for $j=1,\ldots,k$.

Then the inequality $\sum_{r\in \R^q}\psi_B(r)s_r\geq 1$ is extreme
for $R_f(\R^q)$ if and only if the inequality $\sum_{j=1}^k
s_{j}\geq 1$ is extreme for $\conv(R_f(r^1,\ldots,r^{k+h}))$.
\end{theorem}

Even though the data in integer programs are typically rational and
studying the infinite relaxation~(\ref{eq:Rf}) for $W=\Q^q$ seems
natural~\cite{BorCor, cm}, some of its extreme inequalities arise
from maximal lattice-free convex sets that are not rational
polyhedra~\cite{cm}.

For example, the irrational triangle $B$ defined by the inequalities
$x_1 + x_2 \leq 2,\; x_2 \geq 1 + \sqrt{2}x_1,\; x_2 \geq 0$ is a
maximal lattice-free convex set in the plane, and it gives rise to
an extreme valid inequality $\sum_{r\in \Q^2}\psi_B(r)s_r\ge 1$ for
$R_f(\Q^2)$ for any rational $f$ in the interior of $B$. In fact,
{\em every} maximal lattice-free triangle gives rise to an extreme
valid inequality for $R_f(\Q^2)$~\cite{cm}. Therefore, even when $W
= \Q^q$ in (\ref{eq:Rf}), irrational coefficients are needed to
describe some of the extreme inequalities for $R_f(\Q^q)$. Indeed,
it follows from Theorem~\ref{thm:min-ineq-intr} and
from~\cite{BorCor} that the extreme inequalities for $R_f(\Q^q)$ are
precisely the restrictions to $\Q^q$ of the extreme inequalities for
$R_f(\R^q)$. This suggests that the more natural setting
for~\eqref{eq:Rf} might in fact be $W=\R^q$.

\bigskip

The paper is organized as follows. In Section~\ref{sec:max-conv} we
will state and prove the natural extensions of
Theorems~\ref{thm:main-intr} and~\ref{thm:lattice-free-intr} for
general lattices. In Section~\ref{sec:corner} we prove
Theorem~\ref{thm:min-ineq-intr}, while in Section~\ref{sec:closure}
we prove Theorem~\ref{thm:closure} and in Section~\ref{sec:extreme}
we prove Theorem~\ref{thm:extreme}.

\section{Maximal lattice-free convex sets}\label{sec:max-conv}
Given $X\subset \R^n$, we denote by $\langle X\rangle$ the linear
space generated by the vectors in $X$. The underlying field is
$\mathbb{R}$ in this paper. The purpose of this section is to prove
Theorems~\ref{thm:main-intr} and~\ref{thm:lattice-free-intr}. For
this, we will need to work with general lattices.

\begin{definition}
An additive group $\Lambda$ of $\R^n$ is said to be {\em finitely
generated} if there exist vectors $a_1,\ldots,a_m\in\R^n$ such that
$\Lambda=\{\lambda_1 a_1+\ldots+\lambda_m a_m\st
\lambda_1,\ldots,\lambda_m\in \Z\}$.

If a finitely generated additive group $\Lambda$ of $\R^n$ can be
generated by {\em linearly independent} vectors $a_1,\ldots,a_m$,
then  $\Lambda$ is called a  {\em lattice of the linear space}
$\langle a_1,\ldots,a_m \rangle$. The set of vectors
$a_1,\ldots,a_m$ is called a {\em basis} of the lattice $\Lambda$.
\end{definition}

\begin{definition} Let $\Lambda$ be a lattice of a linear space $V$ of $\R^n$. Given a linear subspace $L$ of $V$, we say that $L$ is a {\em $\Lambda$-subspace of $V$} if there exists a basis of $L$ contained in $\Lambda$.
\end{definition}

For example, in $\R^2$, consider the lattice $\Lambda$ generated by
vectors $(0,1)$ and $(1,0)$. The line $x_2 = 2x_1$ is a
$\Lambda$-subspace, whereas the line $x_2 = \sqrt{2}x_1$ is not.

Given $y\in \R^n$ and $\varepsilon>0$, we will denote by $B_\varepsilon(y)$ the open ball centered at $y$ of radius $\varepsilon$.
Given an affine space $W$ of $\R^n$ and a set $S\subseteq W$, we denote by $\intr_W(S)$ the interior of $S$ with respect to the topology induced on $W$ by $\R^n$, namely $\intr_W(S)$ is the set of points $x\in S$ such that $B_\varepsilon(x)\cap W\subset S$ for some $\epsilon>0$. We denote by $\relint(S)$ the relative interior of $S$, that is $\relint(S)=\intr_{\aff(S)}(S)$.

\begin{definition}
Let $\Lambda$ be a lattice of a linear space $V$ of $\R^n$, and let $W$ be a linear space of $\R^n$ containing $V$. A set $S\subset \R^n$ is said to be a {\em $\Lambda$-free convex set of $W$} if $S\subset W$, $S$ is convex and $\Lambda\cap \intr_W(S)=\emptyset$, and $S$ is said to be a {\em maximal $\Lambda$-free convex set of $W$} if it is not properly contained in any $\Lambda$-free convex set.
\end{definition}
\bigskip

The next two theorems are restatements of Theorems~\ref{thm:main-intr} and ~\ref{thm:lattice-free-intr} for general lattices.

\begin{theorem}\label{thm:main} Let $\Lambda$ be a lattice of a linear space $V$ of $\R^n$, and let $W$ be a linear space of $\R^n$ containing $V$. A set $S\subset \R^n$ is a maximal $\Lambda$-free convex set of $W$ if and only if one of the following holds:
\begin{itemize}
\item[(i)] $S$ is a polyhedron in $W$, $\dim(S)=\dim(W)$, $S\cap V$ is a maximal $\Lambda$-free convex set of $V$,
the facets of $S$ and $S\cap V$ are in one-to-one correspondence and
for every facet $F$ of $S$, $F\cap V$ is the facet of $S\cap V$
corresponding to $F$;
\item[(ii)] $S$ is an affine hyperplane of $W$ of the form $S=v+L$ where $v\in S$ and $L\cap V$ is a hyperplane of $V$ that is not a lattice subspace of $V$;
\item[(iii)]  $S$ is a half-space of $W$ that contains $V$ on its boundary.
\end{itemize}
\end{theorem}

\begin{theorem}\label{thm:lattice-free}  Let $\Lambda$ be a lattice of a linear space $V$ of $\R^n$. A set $S\subset \R^n$ is a maximal $\Lambda$-free convex set of $V$ if and only if one of the following holds:
\begin{itemize}
\item[(i)] $S$ is a polyhedron of the form $S=P+L$ where $P$ is a polytope, $L$ is a $\Lambda$-subspace of $V$, $\dim(S)=\dim(P)+\dim(L)=\dim(V)$, $S$ does not contain any point of $\Lambda$ in its interior and there is a point of $\Lambda$ in the relative interior of each facet of $S$;
\item[(ii)] $S$ is an affine hyperplane of $V$ of the form
$S=v+L$ where $v\in S$ and $L$ is not a $\Lambda$-subspace of $V$.
\end{itemize}
\end{theorem}

\subsection{Proof of Theorem~\ref{thm:main}}

We assume Theorem~\ref{thm:lattice-free} holds. Its proof will be given in the next section.\bigskip

\noindent ($\Rightarrow$) Let $S$ be a maximal $\Lambda$-free convex set of $W$. We show that one of $(i)-(iii)$ holds.
If $V=W$, then $(iii)$ cannot occur and either $(i)$ or $(ii)$ follows from Theorem~\ref{thm:lattice-free}. Thus we assume $V\subset W$.

Assume first that $\dim(S)<\dim(W)$. Then there exists a hyperplane
$H$ of $W$ containing $S$, and since $\intr_W(H)=\emptyset$, then
$S=H$ by maximality of $S$. Since $S$ is a hyperplane of $W$, then
either $V\subseteq S$ or $S\cap V$ is a hyperplane of $V$. If
$V\subseteq S$, then let $K$ be one of the two half spaces of $W$
separated by $S$. Then $\intr_W(K)\cap\Lambda=\emptyset$,
contradicting the maximality of $S$. Hence  $S\cap V$ is a
hyperplane of $V$. We show that $P=S\cap V$ is a maximal
$\Lambda$-free convex set of $V$. Indeed, let $K$ be a convex set in
$V$ such that $\intr_V(K)\cap \Lambda=\emptyset$ and $P\subseteq K$.
Since $\conv(S\cup K)\cap V=K$, then $\intr_W(\conv(S\cup
K)\cap\Lambda)=\emptyset$. By maximality of $S$, $S=\conv(S\cup K)$,
hence $P=K$.\\ Given $v\in P$, $S=v+L$ for some hyperplane $L$ of
$W$, and $P=v+(L\cap V)$. Applying Theorem~\ref{thm:lattice-free} to
$P$, we get that $L\cap V$ is not a lattice subspace of $V$, and
case $(ii)$ holds.\medskip

So we may assume $\dim(S)=\dim(W)$. Since $S$ is convex, then $\intr_W(S)\neq\emptyset$. We consider two cases.\medskip

\noindent{\sl Case 1.} $\intr_W(S)\cap V=\emptyset$. \medskip

Since $\intr_W(S)$ and $V$ are nonempty disjoint convex sets, there
exists a hyperplane separating them, i.e. there exist $\alpha\in
\R^n$ and $\beta\in \R$ such that $\alpha x\geq\beta$ for every
$x\in S$ and  $\alpha x\leq\beta$ for every $x\in V$. Since $V$ is a
linear space, then $\alpha x=0$ for every $x\in V$, hence $\beta\geq
0$. Then the half space $H=\{x\in W\st \alpha x\geq 0\}$ contains
$S$ and $V$ lies on the boundary of $H$. Hence $H$ is a maximal
$\Lambda$-free convex set of $W$ containing $S$, therefore $S=H$ by
the maximality assumption, so $(iii)$ holds.
\medskip

\noindent{\sl Case 2.} $\intr_W(S)\cap V\neq\emptyset$.
\medskip

We claim that
\begin{equation}\label{eq:int}\intr_W(S)\cap V=\intr_V(S\cap V).\end{equation}
To prove this claim, notice that the direction $\intr_W(S)\cap
V\subseteq \intr_V(S\cap V)$ is straightforward. Conversely, let
$x\in\intr_V(S\cap V)$. Then there exists $\varepsilon>0$ such that
$B_\varepsilon(x)\cap V\subseteq S$. Since $\intr_W(S)\cap
V\neq\emptyset$, there exists $y\in\intr_W(S)\cap V$. Then there
exists $\varepsilon' > 0$ such that $B_{\varepsilon'}(y)\cap
W\subseteq S$. We may assume $y \neq x$ since otherwise the result
holds. Since $x \in \intr(B_{\varepsilon}(x))$, there exists $z \in
\intr_W(B_{\varepsilon}(x))$ such that $x$ is in the relative interior
of the segment $yz$. Since $x, y \in V$, $z\in V$ and therefore $z
\in S$. The ball $B_\delta(x)$ with radius $\delta =
\varepsilon'\frac{\| x- z\|}{\|y - z\|}$ is contained in
$\conv(\{z\}\cup B_{\varepsilon'}(y))$. By convexity of $S$,
$\conv(\{z\}\cup (B_{\varepsilon'}(y) \cap W)) \subseteq S$ and therefore
$(B_\delta(x) \cap W)\subseteq S$. Thus $x\in\intr_W(S)$. Since $x \in V$,
it follows that $x\in\intr_W(S) \cap V$.

Let $P=S\cap V$. By~(\ref{eq:int}) and because
$\intr_W(S)\cap\Lambda=\emptyset$, we have
$\intr_V(P)\cap\Lambda=\emptyset$. We show that $P$ is a maximal
$\Lambda$-free convex set of $V$. Indeed, let $K$ be a convex set in
$V$ such that $\intr_V(K)\cap \Lambda=\emptyset$ and $P\subseteq K$.
Since $\conv(S\cup K)\cap V=K$, Claim (\ref{eq:int}) implies that
$\intr_W (\conv(S\cup K))\cap \Lambda=\emptyset$. By maximality,
$S=\conv(S\cup K)$, hence $P=K$.

Since  $\dim(P)=\dim(V)$, by Theorem~\ref{thm:lattice-free} applied
to $P$, $P$ is a polyhedron with a point of $\Lambda$ in the
relative interior of each of its facets. Let $F_1,\ldots,F_t$ be the
facets of $P$. For $i=1,\ldots,t$, let $z_i$ be a point in
$\relint(F_i)\cap\Lambda$. By~(\ref{eq:int}), $z_i\notin\intr_W(S)$.
By the separation theorem, there exists a half-space $H_i$ of $W$
containing  $\intr_W(S)$ such that $z_i\notin\intr_W(H_i)$. Notice
that $F_i$ is on the boundary of $H_i$. Then $S\subseteq
\cap_{i=1}^t H_i$. By construction $\intr_W(\cap_{i=1}^t
H_i)\cap\Lambda=\emptyset$, hence by maximality of $S$, $S=
\cap_{i=1}^t H_i$. For every $j=1,\ldots,t$, $\intr_W(\cap_{i\neq j}
H_i)$ contains $z_j$. Therefore $H_j$ defines a facet of $S$ for
$j=1,\ldots,t$.
\bigskip

\noindent ($\Leftarrow$) Let $S$ be a set in $\R^n$ satisfying one
of $(i),(ii),(iii)$. Clearly $S$ is a convex set in $W$ and
$\intr_W(S)\cap \Lambda=\emptyset$, so we only need to prove
maximality. If $S$ satisfies $(iii)$, then this is immediate. This
is also immediate when $S$ satisfies $(i)$ and $V=W$. So we may assume that either
$S$ satisfies (i) and $\dim(V) < \dim(W)$, or $S$ satisfies $(ii)$.

Suppose that there exists a closed
convex set $K\subset W$ strictly containing $S$ such that
$\intr_W(K)\cap \Lambda=\emptyset$. Let $w\in K\sm S$. Then
$\conv(S\cup\{w\})\subseteq K$. To conclude the proof of the
theorem, it suffices to prove that
$S\cap V$ is strictly contained in $\conv(S\cup\{w\})\cap V$. Indeed, by
maximality of $S \cap V$, this claim implies that the set
$\intr_V(\conv(S\cup \{ w \})\cap V)$ contains a point in $\Lambda$.
Now $\conv(S\cup\{w\})\subseteq K$ implies that $\intr_W(K)$
contains a point of $\Lambda$, a contradiction.

It only remains to prove that  $S\cap V\subset
\conv(S\cup\{w\})\cap V$. This is clear when $S$ is a
hyperplane satisfying (ii). The statement is also clear if $w\in V$. Assume now that $S$ is a polyhedron
satisfying (i) and $\dim(V) < \dim(W)$ and that $w\notin V$.
Let $F$ be a facet of $S$ that separates $w$ from
$S$. If $F\cap V$ is contained in a proper face of $F$, then $F\cap
V$ is contained in at least two facets of $S$, a contradiction to
the one-to-one correspondence property. So $F \cap V$ is not
contained in  proper face of $F$. Therefore, there exists $p \in
\relint(F)\cap V$. Choose $\varepsilon>0$ such that
$B_\varepsilon(p) \cap \aff(F) \subseteq F$. Note that $F
\not\subseteq V$ since otherwise $F\cap V = F$ but this is a
contradiction since $\dim(S \cap V) = \dim(V) < \dim(W) = \dim(S)$,
and $\dim(F\cap V) = \dim(S\cap V) - 1$, $\dim(F) = \dim(S) - 1$.

Let $W'=\aff(V\cup \{w\})$. Note that  $V$ and $\aff(F)\cap W'$ are distinct affine hyperplanes of $W'$. Let $H, H'$ be the two open half-spaces of $W'$ defined by $V$, and assume w.l.o.g. that $w\in H'$. Since $p\in V$, $H\cap \aff(F)\cap B_\varepsilon(p)$ contains some point $t$. Since $B_\varepsilon(p) \cap \aff(F) \subseteq F$, it follows that $t\in H\cap F$. Let $T$ be the line segment joining $w$ and $t$. Since $t\in H$ and $w\in H'$ it follows that $T\cap V$ contains exactly one point, say $\bar w$. Note that $\bar w\neq t,w$. Since $w\notin S$ and $t\in F$, we have that $\bar w\in \conv(S\cup \{w\})\cap V$ but $\bar w\notin S$.
 \qed

\subsection{Proof of Theorem~\ref{thm:lattice-free}}

To simplify notation, given $S\subseteq \R^n$, we denote
$\intr_V(S)$ simply by $\intr(S)$.

The following standard result in lattice theory provides a useful
equivalent definition of lattice (see Barvinok~\cite{barv}, p. 284
Theorem~1.4).

\begin{theorem}\label{thm:discrete}
Let $\Lambda$ be the additive group generated by vectors
$a_1,\ldots,a_m\in\R^n$. Then $\Lambda$ is a lattice of the linear
space $\langle a_1,\ldots,a_m\rangle$ if and only if there exists
$\varepsilon>0$ such that $\|y\|\geq \varepsilon$ for every
$y\in\Lambda \setminus \{ 0 \}$.
\end{theorem}

In this paper we will only need the ``only if'' part of the
statement, which is easy to prove (see~\cite{barv}, p. 281 problem
5). Theorem~\ref{thm:discrete} implies the following result
(see~\cite{barv}, p. 281 problem 3).

\begin{corollary}\label{lemma:discrete}
Let $\Lambda$ be a lattice of a linear space of $\R^n$. Then every
bounded set in $\R^n$ contains a finite number of points in
$\Lambda$.
\end{corollary}

Throughout this section, $\Lambda$ will be a lattice of a linear space
$V$ of $\R^n$. The following lemma proves the ``only if'' part of
Theorem~\ref{thm:lattice-free}  when $S$ is bounded and
full-dimensional.

\begin{lemma}\label{lemma:bounded}
Let $S\subset V$ be a bounded maximal $\Lambda$-free convex set with
$\dim(S)=\dim(V)$. Then $S$ is a polytope with a point of $\Lambda$
in the relative interior of each of its facets.
\end{lemma}

\begin{proof}
Since $S$ is bounded, there exist integers $L, U$ such that $S$ is
contained in the box $B=\{x\in \R^n\st L\leq x_i\leq U\}$. For each
$y\in\Lambda\cap B$, since $S$ is convex there exists a closed
half-space $H^y$ of $V$ such that $S\subseteq H^y$ and
$y\notin\intr(H^y)$. By Corollary~\ref{lemma:discrete}, $B\cap
\Lambda$ is finite, therefore $\bigcap_{y\in B\cap \Lambda} H^y$ is
a polyhedron. Thus $P= \bigcap_{y\in B\cap \Lambda} H^y \cap B$ is a
polytope and by construction $\Lambda\cap\intr(P)=\emptyset$. Since
$S\subseteq B$ and $S\subseteq H^y$ for every $y\in B\cap\Lambda$,
it follows that $S\subseteq P$. By maximality of $S$, $S=P$,
therefore $S$ is a polytope. We only need to show that $S$ has a
point of $\Lambda$ in the relative interior of each of its facets.
Let $F_1,\ldots,F_t$ be the facets of $S$, and let $H_i=\{x\in V\st
\alpha_i x\leq \beta_i\}$ be the closed half-space defining $F_i$,
$i=1,\ldots,t$. Then $S=\bigcap_{i=1}^t H_i$. Suppose, by
contradiction, that one of the facets of $S$, say $F_t$, does not
contain a point of $\Lambda$ in its relative interior. Given
$\varepsilon>0$, the polyhedron $S'=\{x\in V\st \alpha_i x\leq
\beta_i,\,i=1,\ldots,t-1,\, \alpha_t x\leq \beta_t+\varepsilon\}$ is
a polytope since it has the same recession cone as $S$. The polytope
$S'$ contains points of $\Lambda$ in its interior  by the maximality
of $S$. By Corollary~\ref{lemma:discrete}, $\intr(S')$ has a finite
number of points in $\Lambda$, hence there exists one minimizing
$\alpha_t x$, say $z$. By construction, the polytope $S'=\{x\in V\st
\alpha_i x\leq \beta_i,\,i=1,\ldots,t-1,\, \alpha_t x\leq \alpha_t
z\}$ does not contain any point of $\Lambda$ in its interior and
properly contains $S$, contradicting the maximality of $S$.
\end{proof}
\bigskip

We will also need the following famous theorem of Dirichlet.

\begin{theorem}[Dirichlet] Given real numbers $\alpha_1,\ldots,\alpha_n,\varepsilon$ with $0<\varepsilon<1$, there exist integers $p_1,\ldots,p_n$ and $q$ such that
\begin{equation}
\left|{\alpha_i-\frac {p_i} q}\right|<\frac\varepsilon q,\, \mbox{ for }\, i=1,\ldots, n, \, \mbox{ and } \, 1\leq q\leq \varepsilon^{-1}.
\end{equation}
\end{theorem}

The following is a  consequence of Dirichlet's theorem.

\begin{lemma}\label{lemma:half-line}
Given $y\in \Lambda$ and $r\in V\setminus\{0\}$, then for every
$\varepsilon
>0$ and $\bar\lambda\geq 0$, there exists a  point of
$\Lambda\setminus \{y\}$ at distance less than $\varepsilon$ from
the half line $\{y+\lambda r\st \lambda\geq \bar\lambda\}.$
\end{lemma}
\begin{proof} First we show that, if the statement holds for $\bar\lambda = 0$, then it holds for arbitrary $\bar\lambda$.
Given $\varepsilon>0$, let $Z$ be the set of points of $\Lambda$ at
distance less than $\varepsilon$ from $\{y+\lambda r\st \lambda\geq
0\}$. Suppose, by contradiction, that no point in $Z$ has distance
less than $\varepsilon$ from $\{y+\lambda r\st \lambda\geq
\bar\lambda\}$. Then $Z$ is contained in
$B_\varepsilon(0)+\{y+\lambda r\st 0\leq \lambda\leq \bar\lambda\}$.
By Corollary~\ref{lemma:discrete}, $Z$ is finite, thus there exists
an $\bar\varepsilon>0$ such that every point in $Z$ has distance
greater than $\bar\varepsilon$ from $\{y+\lambda r\st \lambda\geq
0\}$, a contradiction. So we only need to show that, given
$\varepsilon > 0$, there exists at least one point of
$\Lambda\setminus\{y\}$ at distance at most $\varepsilon$ from
$\{y+\lambda r\st \lambda\geq 0\}$. We may assume $\varepsilon < 1$.

Without loss of generality, assume $\|r\| = 1$. Let $m=\dim(V)$ and
$a_1,\ldots,a_m$ be a basis of $\Lambda$. Then there exists
$\alpha\in\R^m$ such that $r=\alpha_1 a_1+\ldots +\alpha_m a_m$.
Denote by $A$ the matrix with columns $a_1,\ldots,a_m$, and define
$\|A\|=\sup_{x\,:\,\|x\|\leq 1} \|Ax\|$ where, for a vector $v$,
$\|v\|$ denotes the Euclidean norm of $v$. Choose $\delta>0$ such
that $\delta<1$ and $\delta\leq\varepsilon/(\|A\|\sqrt{m})$. By
Dirichlet's theorem, there exist $p\in\Z^m$ and $\lambda\geq 1$ such
that
$$\|\alpha-\frac p\lambda\|=\sqrt{\sum_{i=1}^m{\left|{\alpha_i-\frac
{p_i} \lambda}\right|}^2}\leq
\frac{\delta\sqrt{m}}{\lambda}\leq\frac{\varepsilon}{\|A\|\lambda}.$$

Let $z=Ap+y$. Since $p\in\Z^m$, then $z\in\Lambda$. Note that $p
\neq 0$ since $\|\alpha\| \geq \frac{\|A\alpha\|}{\|A\|} =
\frac{\|r\|}{\|A\|} > \frac{\varepsilon}{\|A\|\lambda}$, where the
first inequality follows from the definition of $\|A\|$ and the last
one follows from the assumptions on $\|r\|$, $\varepsilon$ and
$\lambda$. Therefore, $z \in \Lambda\setminus\{y\}$. Furthermore
$$\|(y+\lambda r) -z\|=\|\lambda r-Ap\|=\|A(\lambda \alpha-p)\|\leq \|A\| \|\lambda\alpha-p\|\leq \varepsilon. $$
\end{proof}
\bigskip

\begin{lemma}\label{lemma:recession-cone} Let $S$ be a $\Lambda$-free convex set, and let $C=\rec(S)$. Then also $S+\langle C\rangle $ is $\Lambda$-free.
\end{lemma}
\begin{proof}
Let $r\in C$, $r\neq 0$. We only need to  show that $S+\langle r\rangle$ is $\Lambda$-free. Suppose there exists $y\in \intr(S+\langle r\rangle)\cap \Lambda$.
We show that $y \in \intr(S) + \langle r\rangle$. Suppose not. Then $(y + \langle r\rangle) \cap \intr(S) = \emptyset$, which implies that there is a hyperplane $H$ separating the line $y + \langle r\rangle$ and $S + \langle r\rangle$. This contradicts $y\in \intr(S+\langle r\rangle )$.
This shows $y \in \intr(S) + \langle r\rangle$. Thus there exists  $\bar\lambda$ such that  $\bar y=y+\bar\lambda r\in \intr(S)$, i.e. there exists $\varepsilon>0$ such that $B_\varepsilon(\bar y)\cap V\subset S$. Since $y\in\Lambda$, then $y\notin\intr(S)$, and thus, since $\bar y\in \intr(S)$ and $r\in C$, we must have $\bar\lambda>0$.
Since $r\in C$, then  $B_\varepsilon(\bar y)+\{\lambda r\st \lambda\geq 0\}\subset S$. Since $y\in\Lambda$, by Lemma~\ref{lemma:half-line} there exists $z\in \Lambda$ at distance less than $\varepsilon$ from the half line $\{y+\lambda r\st \lambda\geq \bar\lambda\}$. Thus $z\in B_\varepsilon(\bar y)+\{\lambda r\st \lambda\geq 0\}$, hence $z\in\intr(S)$, a contradiction.
\end{proof}

Given a linear subspace $L$ of $\R^n$, we denote by $L^\bot$ the orthogonal complement of $L$. Given a set $S\subseteq\R^n$, the {\em orthogonal projection of $S$ onto $L^\bot$} is the set
$$\proj_{L^\bot}(S)=\{v\in L^\bot\st v+w\in S\, \mbox{ for some } w\in L\}.$$

We will use the following result (see Barvinok~\cite{barv}, p. 284 problem 3).

\begin{lemma}\label{lemma:rat-proj}
Given a $\Lambda$-subspace $L$ of $V$, the orthogonal projection of
$\Lambda$ onto $L^\bot$ is a lattice of $L^\bot\cap V$.
\end{lemma}

\begin{lemma}\label{lemma:irrat-space}
If a linear subspace $L$ of $V$ is not a $\Lambda$-subspace of $V$,
then for every $\varepsilon>0$ there exists $y\in\Lambda\sm L$ at
distance less than $\varepsilon$ from $L$.
\end{lemma}
\begin{proof} The proof is by induction on $k=\dim(L)$. Assume $L$ is a linear subspace of $V$ that is not a $\Lambda$-subspace, and let $\varepsilon>0$.  If $k=1$, then, since the origin $0$ is contained in $\Lambda$, by Lemma~\ref{lemma:half-line} there exists $y\in \Lambda$ at distance
less than $\varepsilon$ from $L$. If $y\in L$, then $L=\langle y\rangle$, thus $L$ is a $\Lambda$-subspace of $V$, contradicting our assumption. \\
Hence we may assume that $k\geq 2$ and the statement holds for spaces of dimension $k-1$.
\bigskip

{\em Case 1:} $L$ contains a nonzero vector $r\in\Lambda$. Let
$$L'=\proj_{\langle r\rangle^\bot}(L),\quad \Lambda'=\proj_{\langle
r\rangle^\bot}(\Lambda).$$
By Lemma~\ref{lemma:rat-proj}, $\Lambda'$ is a lattice of $\langle r\rangle^\bot\cap V$. Also,
$L'$ is not a lattice subspace of $\langle r\rangle^\bot\cap V$ with respect to $\Lambda'$,
because if there exists a basis $a_1,\ldots, a_{k-1}$ of $L'$ contained in $\Lambda'$, then
there exist scalars $\mu_1,\ldots,\mu_{k-1}$ such that $a_1+\mu_1 r,\ldots,a_{k-1}+\mu_{k-1} r\in\Lambda$,
but then $r, a_1+\mu_1 r,\ldots,a_{k-1}+\mu_{k-1} r$ is a basis of $L$ contained in $\Lambda$, a contradiction.
By induction, there exists a point $y'\in\Lambda'\sm L'$ at distance less than $\varepsilon$ from $L'$.
Since $y'\in\Lambda'$, there exists a scalar $\mu$ such that $y=y'+\mu r\in\Lambda$, and $y$
has distance less than $\varepsilon$ from $L$.
\bigskip

{\em Case 2:} $L\cap\Lambda=\{0\}$. By Lemma~\ref{lemma:half-line},
there exists a nonzero vector $y\in\Lambda$ at distance less than
$\varepsilon$ from $L$. Since $L$ does not contain any point in
$\Lambda$ other than the origin, $y\notin L$.
\end{proof}

\begin{lemma}\label{lemma:hyperplane}
Let $L$ be a linear subspace of $V$ with $\dim(L)=\dim(V)-1$, and let $v\in V$. Then $v+L$ is a maximal $\Lambda$-free convex set if and only if $L$ is not a lattice subspace of $V$.
\end{lemma}
\begin{proof}
$(\Rightarrow)$ Let $S=v+L$ and assume that $S$ is a maximal
$\Lambda$-free convex set. Suppose by contradiction that $L$ is a
$\Lambda$-subspace. Then there exists a basis $a_1,\ldots,a_m$ of
$\Lambda$ such that $a_1, \ldots, a_{m-1}$ is a basis of $L$. Thus
$S=\{\sum_{i=1}^m x_i a_i\,|\,  x_m=\beta\}$ for some $\beta \in
\mathbb{R}$. Then,  $K=\{\sum_{i=1}^m x_i a_i\,|\,
\ceil{\beta-1}\leq  x_m\leq\ceil{\beta}\}$ strictly contains $S$ and
$\intr(K)\cap\Lambda=\emptyset$, contradicting the maximality of
$S$.

$(\Leftarrow)$ Assume $L$ is not a $\Lambda$-subspace of $V$. Since
$S=v+L$ is an affine hyperplane of $V$, $\intr(S)=\emptyset$, thus
$\intr(S)\cap\Lambda=\emptyset$, hence we only need to prove that
$S$ is maximal with such property. Suppose not, and let $K$ be a
maximal convex set in $V$ such that $\intr(K)\cap\Lambda=\emptyset$
and $S\subset K$. Then by maximality $K$ is closed. Let $w\in K\sm
S$. Since $K$ is closed and convex, $\ol\conv(\{w\}\cup S) \subseteq
K$. Since $\ol\conv(\{w\}\cup S) = \ol\conv(\{w\}\cup (v + L)) =
\conv(\{v,w\})+L$, we have that $K\supseteq \conv(\{v,w\})+L$. Let
$\varepsilon$ be the distance between $v+L$ and $w+L$, and $\delta$
be the distance of $\conv(\{v,w\})+L$ from the origin. By
Lemma~\ref{lemma:irrat-space}, since $L$ is not a $\Lambda$-subspace
of $V$, there exists a vector $y\in\Lambda\sm L$ at distance
$\bar\varepsilon<\varepsilon$ from $L$. Moreover, either $y$ or $-y$
has distance strictly less than $\delta$ from $\conv(\{v,w\})+L$. We
conclude that either $(\floor{\frac \delta{\bar\varepsilon}}+1)y$ or
$-(\floor{\frac \delta{\bar\varepsilon}}+1)y$ is strictly between
$v+L$ and $w+L$, and therefore is in the interior of $K$. Since
these two points are integer multiples of $y\in\Lambda$, this is a
contradiction.
\end{proof}
\medskip

We are now ready to prove Lov\'asz's Theorem.

\begin{proof}[Proof of Theorem~\ref{thm:lattice-free}.]
$(\Leftarrow)$ If $S$ satisfies $(ii)$, then by
Lemma~\ref{lemma:hyperplane}, $S$ is a maximal $\Lambda$-free convex
set. If $S$ satisfies $(i)$, then, since
$\intr(S)\cap\Lambda=\emptyset$, we only need to show that $S$ is
maximal. Suppose not, and let $K$ be a convex set in $V$ such that
$\intr(K)\cap\Lambda=\emptyset$ and $S\subset K$. Given $y\in K\sm
S$, there exists a hyperplane $H$ separating $y$ from $S$ such that
$F=S\cap H$ is a facet of $S$.  Since $K$ is convex and $S\subset
K$, then $\conv(S\cup \{y\})\subseteq K$. Since $\dim(S)=\dim(V)$,
$F\subset S$ hence the $\relint(F)\subset\intr(K)$. By assumption,
there exists $x\in\Lambda\cap\relint(F)$, so $x\in\intr(K)$, a
contradiction.

\bigskip

$(\Rightarrow)$ Let $S$ be a maximal $\Lambda$-free convex set. We show that $S$ satisfies either $(i)$ or $(ii)$.
Observe that, by maximality, $S$ must be closed.\medskip

If $\dim(S)<\dim(V)$, then $S$ is contained in some affine hyperplane $H$. Since $\intr(H)=\emptyset$, we have $S=H$ by maximality of $S$, therefore $S=v+L$ where $v\in S$ and $L$ is a hyperplane in $V$. By Lemma~\ref{lemma:hyperplane}, $(ii)$ holds.
\medskip

Therefore we may assume that $\dim(S)=\dim(V)$. In particular, since $S$ is convex, $\intr(S)\neq\emptyset$.
By Lemma~\ref{lemma:bounded}, if $S$ is bounded, $(i)$ holds. Hence we may assume that $S$ is unbounded. Let $C$ be the recession cone of $S$ and $L$ the lineality space of $S$. By standard convex analysis, $S$ is unbounded if and only if $C\neq\{0\}$ (see for example Proposition~2.2.3 in~\cite{lemar}).
\bigskip

\noindent{\bf Claim 1.} {\em  $L=C$. }
\medskip

By Lemma~\ref{lemma:recession-cone}, $S+\langle C\rangle$ is $\Lambda$-free. By maximality of $S$ this implies that $S=S+\langle C\rangle$, hence $\langle C\rangle \subseteq L$. Since $L\subseteq \langle C\rangle$, it follows that $L=C$. \hfill $\diamond$
\bigskip

Let $P=\proj_{L^\bot}(S)$ and $\Lambda'=\proj_{L^\bot}(\Lambda)$. By Claim~1, $S=P+L$ and
$P \subset L^\bot\cap V$ is a bounded set. Furthermore, $\dim(S) = \dim(P)+\dim(L)=\dim(V)$ and $\dim(P)=\dim(L^\bot\cap V)$. Notice that $\intr(S)=\relint(P)+L$, hence $\relint(P)\cap\Lambda'=\emptyset$. Furthermore $P$ is inclusionwise maximal among the convex sets of $L^\bot\cap V$ without points of $\Lambda'$ in the relative interior: if not,  given a convex set $K\subseteq L^\bot\cap V$ strictly containing $P$ and with no point of $\Lambda'$ in its relative interior, we have $S=P+L\subset K+L$, and $K+L$ does not contain any point of $\Lambda$ in its interior, contradicting the maximality of $S$.
\bigskip

\noindent{\bf Claim 2.} {\em   $L$ is a $\Lambda$-subspace of $V$.}
\medskip

By contradiction, suppose $L$ is not a $\Lambda$-subspace of $V$.
Then, by Lemma~\ref{lemma:irrat-space}, for every $\varepsilon>0$,
there exists a point in $\Lambda\setminus L$ whose distance from $L$
is at most $\varepsilon$. Therefore, its projection onto $L^{\bot}$
is a point $y\in\Lambda'\setminus\{0\}$ such that
$\|y\|<\varepsilon$. Let $V_\varepsilon$ be the linear subspace of
$L^\bot\cap V$ generated by the points in $\{y\in\Lambda'\st
\|y\|<\varepsilon\}$. Then $\dim(V_\varepsilon)>0$.

Notice that, given $\varepsilon'>\varepsilon''>0$, then $V_{\varepsilon'}\supseteq V_{\varepsilon''}\supset\{0\}$, hence there exists $\varepsilon_0>0$ such that $V_\varepsilon=V_{\varepsilon_0}$ for every $\varepsilon<\varepsilon_0$. Let $U=V_{\varepsilon_0}$.

By definition, $\Lambda'$ is dense in $U$ (i.e. for every $\varepsilon>0$ and every $x\in U$ there exists $y\in \Lambda'$ such that $\|x-y\|<\varepsilon$). Thus, since $\relint(P)\cap\Lambda'=\emptyset$, we also have $\relint(P)\cap U=\emptyset$. Since $\dim(P)=\dim(L^\bot\cap V)$, it follows that  $\relint(P)\cap (L^\bot\cap V)\neq\emptyset$, so in particular  $U$ is a proper subspace of $L^\bot\cap V$.

Let $Q=\proj_{(L+U)^\bot}(P)$ and $\Lambda''=\proj_{(L+U)^\bot}(\Lambda')$. We show that $\relint(Q)\cap\Lambda''=\emptyset$. Suppose not, and let $y\in\relint(Q)\cap\Lambda''$. Then,  $y+w\in\Lambda'$ for some $w\in U$. Furthermore, we claim that $y+w'\in\relint(P)$ for some $w'\in U$. Indeed, suppose no such $w'$ exists. Then $(y+ U)\cap (\relint(P)+U)=\emptyset$. So there exists a hyperplane $H$ in $L^\bot\cap V$ separating $y+U$ and $P+U$. Therefore the projection of $H$ onto $(L+U)^\bot$ separates $y$ and $Q$, contradicting $y \in \relint(Q)$. Thus
$z=y+w'\in\relint(P)$ for some $w'\in U$.
Since $z\in\relint(P)$, there exists $\bar\varepsilon>0$ such that  $B_{\bar\varepsilon}(z)\cap (L^\bot\cap V)\subset \relint(P)$. Since $\Lambda'$ is dense in $U$ and $y+w \in \Lambda'$, it follows that $\Lambda'$ is dense in $y+U$.
Hence, since $z\in y+U$, there exists $\bar x\in\Lambda'$ such that $\|\bar x- z\|<\bar\varepsilon$, hence $\bar x\in\relint(P)$, a contradiction. This shows $\relint(Q)\cap\Lambda''=\emptyset$.

Finally, since $\relint(Q)\cap\Lambda''=\emptyset$, then $\intr(Q+L+U)\cap\Lambda=\emptyset$. Furthermore $P\subseteq Q+U$, therefore $S\subseteq Q+L+U$. By the maximality of $S$, $S=Q+L+U$ hence the lineality space of $S$ contains $L+U$, contradicting the fact that $L$ is the lineality space of $S$ and $U\neq\{0\}$. \hfill $\diamond$
\bigskip

Since $L$ is a $\Lambda$-subspace of $V$,  $\Lambda'$ is a lattice
of $L^\bot\cap V$ by Lemma~\ref{lemma:rat-proj}. Since $P$ is a
bounded maximal $\Lambda'$-free convex set, it follows from
Lemma~\ref{lemma:bounded} that $P$ is a polytope with a point of
$\Lambda'$ in the relative interior of each of its facets, therefore
$S=P+L$ has a point of $\Lambda$ in the relative interior of each of
its facets, and $(i)$ holds.
\end{proof}

From the proof of Theorem~\ref{thm:lattice-free} we get the following.

\begin{corollary}\label{cor:Bexists} Every $\Lambda$-free convex set of $V$ is contained in some maximal $\Lambda$-free convex set of $V$.
\end{corollary}
\begin{proof} Let $S$ be a $\Lambda$-free convex set of $V$. If $S$ is bounded, the proof of Lemma~\ref{lemma:bounded} shows that the corollary holds. If $S$ is unbounded, Claim~1 in the proof of Theorem~\ref{thm:lattice-free} shows that $S+\langle C\rangle$ is $\Lambda$-free, where $C$ is the recession cone of $S$. Hence we may assume that the lineality space $L$ of $S$ is equal to the recession cone of $S$. The projection $P$ of $S$ onto $L^\bot$ is bounded. If $L$ is a $\Lambda$-subspace, then $\Lambda'=\proj_{L^\bot}\Lambda$ is a lattice and $P$ is $\Lambda'$-free, hence it is contained in a maximal $\Lambda'$-free convex set $B$ of $L^\bot\cap V$, and $B+L$ is a maximal $\Lambda$-free convex set of $V$ containing $S$. If $L$ is not a $\Lambda$-subspace, then we may define a linear subspace $U$ of $L^\bot\cap V$ and sets $Q$ and $\Lambda''$ as in the proof of Claim~2. Then proof of Claim~2 shows that $Q$ is a bounded $\Lambda''$-free convex set of $V\cap (L+U)^\bot$ and $\Lambda''$ is a lattice, thus $Q$ is contained in a maximal $\Lambda''$-free convex set $B$ of $V\cap (L+U)^\bot$, and $B+(L+U)$ is a maximal $\Lambda$-free convex set of $V$ containing $S$.
\end{proof}

\section{Minimal Valid Inequalities}\label{sec:corner}

In this section we will prove Theorem~\ref{thm:min-ineq-intr}.
For ease of notation, we denote $R_f(W)$ simply by $R_f$ in this section.
A linear function $\Psi:\,\mathcal W\rightarrow \R$ is of the form
\begin{equation}\label{eq:def-linear}\Psi(s)=\sum_{r\in W}\psi(r) s_r,\quad  s\in\mathcal W\end{equation}
for some $\psi\,:\,W\rightarrow \R$. Throughout the rest of the paper, capitalized Greek letters indicate linear functions from $\W$ to $\R$, while the corresponding lowercase letters indicate functions from $W$ to $\R$ as defined in~(\ref{eq:def-linear}).
\medskip

\begin{definition} A function $\sigma\,:\, W\rightarrow \R$ is {\em positively
homogeneous} if $\sigma(\lambda r)=\lambda\sigma(r)$ for every $r\in
W$ and scalar $\lambda\geq 0$, and it is {\em subadditive} if
$\sigma(r^1+r^2)\leq \sigma(r^1)+\sigma(r^2)$ for every $r^1,r^2\in
W$. The function $\sigma$ is {\em sublinear} if it is positively
homogeneous and subadditive.
\end{definition}

Note that if $\sigma$ is sublinear, then $\sigma(0)=0$. One can
easily show that a function is sublinear if and only if it is
positively homogeneous and convex. We also recall that convex
functions are continuous on their domain, so if $\sigma$ is
sublinear it is also continuous~\cite{lemar}.

\begin{definition} Inequality $\sum_{r\in W}\psi(r)s_r\geq \alpha$
{\em dominates} inequality $\sum_{r\in W}\psi'(r)s_r\geq \alpha$ if
$\psi(r) \leq \psi'(r)$ for all $r\in W$.
\end{definition}

\begin{lemma}\label{lemma:sublinear} Let $\Psi(s)\geq \alpha$ be a valid linear inequality for $R_f$. Then  $\Psi(s)\geq \alpha$ is dominated by a valid linear inequality  $\Psi'(s)\geq \alpha$ for $R_f$ such that $\psi'$ is sublinear.
\end{lemma}
\begin{pf} We first prove the following.
\medskip

\noindent{\bf Claim 1.} {\em For every $s\in\W$ such that $\sum_{r\in W}rs_r=0$ and $s_r\geq 0$, $r\in W$, we have $\sum_{r\in W}\psi(r)s_r\geq 0$.}
\medskip

Suppose not. Then there exists $s\in \W$ such that $\sum_{r\in W}rs_r=0$, $s_r\geq 0\mbox{ for all } r\in W$ and $\sum_{r\in W}\psi(r)s_r<0$. Let $\bar x$ be an integral point in $W$. For any $\lambda>0$, we define $s^\lambda\in \W$ by
$$s^\lambda_r=\left\{
\begin{array}{ll}
1+\lambda s_r &\mbox{ for } r=\bar x-f\\
\lambda s_r &\mbox{ otherwise. }
\end{array}
\right.$$
Since $f+\sum_{r\in W}rs^\lambda_r=\bar x$, it follows that $s^\lambda$ is in $R_f$. Furthermore $\sum_{r\in W}\psi(r)s^\lambda_r=\psi(\bar x-f)+\lambda(\sum_{r\in W}\psi(r)s_r)$. Therefore $\sum_{r\in W}\psi(r)s^\lambda_r$ goes to $-\infty$ as $\lambda$ goes to $+\infty$.\hfill $\diamond$
\bigskip

We define, for all $\bar r\in W$,
$$\psi'(\bar r)=\inf\{\sum_{r\in W}\psi(r)s_r\st \bar r=\sum_{r\in W}rs_r,\, s\in\W,\, s_r\geq 0\mbox{ for all } r\in W\}.$$

By Claim~1, $\sum_{r\in W}\psi(r)s_r\geq -\psi(-\bar r)$ for all $s\in\W$ such that $\bar r=\sum_{r\in W}rs_r$ and  $s_r\geq 0$ for all $r\in W$. Thus the infimum in the above equation is finite and the function $\psi'$ is well defined. Note also that $\psi'(\bar r)\leq\psi(\bar r)$ for all $\bar r\in W$, as follows by considering $s\in \W$ defined by $s_{\bar r}=1$, $s_r=0$ for all $r\in W$, $r\neq \bar r$.
\bigskip

\noindent{\bf Claim 2.} {\em The function $\psi'$ is sublinear}
\medskip

Note first that $\psi'(0)=0$. Indeed, Claim~1 implies $\psi'(0)\geq 0$, while choosing $s_r=0$ for all $r\in W$ shows $\psi'(0)\leq 0$.

Next we show that $\psi'$ is positively homogeneous. To prove this, let $\bar r\in W$ and $s\in\W$ such that $\bar r=\sum_{r\in W}rs_r$ and  $s_r\geq 0$ for all $r\in W$. Let $\gamma=\sum_{r\in W}\psi(r)s_r$. For every $\lambda>0$, $\lambda\bar r=\sum_{r\in W}r (\lambda s_r)$,  $\lambda s_r\geq 0$ for all $r\in W$, and $\sum_{r\in W}\psi(r)(\lambda s_r)=\lambda\gamma$. Therefore $\psi'(\lambda \bar r)=\lambda\psi'(r)$.

Finally, we show that $\psi'$ is convex. Suppose by contradiction that there exist $r',r''\in W$ and $0<\lambda<1$ such that $\psi'(\lambda r'+(1-\lambda)r'')>\lambda\psi'(r')+(1-\lambda)\psi'(r'')+\epsilon$ for some positive $\epsilon$. By definition of $\psi'$, there exist $s',s''\in \W$ such that $r'=\sum_{r\in W}rs'_r$, $r''=\sum_{r\in W}rs''_r$, $s'_r,s''_r\geq 0$ for all $r\in W$, $\sum_{r\in W}\psi(r)s'_r<\psi'(r')+\epsilon$ and
$\sum_{r\in W}\psi(r)s''_r<\psi'(r'')+\epsilon$. Since $\sum_{r\in W} r(\lambda s'_r+(1-\lambda)s''_r)=\lambda r'+(1-\lambda)r''$, it follows that
$\psi'(\lambda r'+(1-\lambda)r'')\leq \sum_{r\in W}\psi(r)(\lambda s'_r+(1-\lambda)s''_r)<\lambda\psi'(r')+(1-\lambda)\psi'(r'')+\epsilon$, a contradiction.\hfill $\diamond$
\bigskip

\noindent{\bf Claim 3.} {\em The inequality $\sum_{r\in W}\psi'(r)s_r\geq \alpha$ is valid for $R_f$.}
\medskip

Suppose there exists $\bar s\in R_f$ such that $\sum_{r\in W}\psi'(r)\bar s_r\leq  \alpha-\epsilon$ for some positive $\epsilon$. Let $\{r^1,\ldots,r^k\}=\{r\in W\st \bar s_r>0\}$. For every $i=1,\ldots,k$, there exists $s^i\in W$ such that $r^i=\sum_{r\in W}rs^i_r$, $s^i_r\geq 0$, $r\in W$, and $\sum_{r\in W}\psi(r)s^i_r<\psi'(r^i)+\epsilon/(k\bar s_{r^i})$.

Let $\tilde s=\sum_{i=1}^k \bar s_{r^i}s^i$. Then
$$\sum_{r\in W}r\tilde s_r=\sum_{r\in W}\sum_{i=1}^k r \bar s_{r^i}s^i_r=\sum_{i=1}^k \bar s_{r^i}\sum_{r\in W}r s^i_r=\sum_{i=1}^k r^i\bar s_{r^i}=\sum_{r\in W}r\bar s_r,$$
hence $\tilde s\in R_f$. Therefore $\sum_{r\in W}\psi(r)\tilde s_r\geq \alpha$ since $\sum_{r\in W}\psi(r)s_r\geq \alpha$ is valid for $R_f$. Now
\begin{eqnarray*}
\sum_{r\in W}\psi(r)\tilde s_r &=&\sum_{r\in W}\sum_{i=1}^k \psi(r) \bar s_{r^i}s^i_r=\sum_{i=1}^k  \bar s_{r^i}\sum_{r\in W}\psi(r)s^i_r\\
&<&\sum_{i=1}^k  \bar s_{r^i}(\psi'(r^i)+\epsilon/(k\bar s_{r^i}))
=\sum_{r\in W}  \psi'(r^i)\bar s_{r^i}+\epsilon\leq \alpha,
\end{eqnarray*}
a contradiction.
\end{pf}

Recall the following definitions from the introduction.

\begin{definition} A valid inequality $\sum_{r\in W}\psi(r)s_r\geq \alpha$ for
$R_f$ is {\em minimal} if it is not dominated by any valid linear
inequality $\sum_{r\in W}\psi'(r)s_r\geq\alpha$ for $R_f$ such that
$\psi'\neq\psi$.
\end{definition}

\begin{definition} Let $V$ be the affine hull of $(f+W)\cap \Z^q$.
Let $C\in\R^{\ell\times q}$ and $d\in\R^\ell$ be such that $V=\{x\in
f+W\st Cx=d\}$. Given two valid inequalities $\sum_{r\in W}
\psi(r)s_r\geq \alpha$ and $\sum_{r\in W} \psi'(r)s_r\geq \alpha'$
for $R_f(W)$, we say that they are {\em equivalent} if there exist
$\rho>0$ and $\lambda\in\R^\ell$ such that
$\psi(r)=\rho\psi'(r)+\lambda^T Cr$ and
$\alpha=\rho\alpha'+\lambda^T(d-Cf)$.
\end{definition}

\begin{lemma}\label{lemma:equiv} Let $\Psi(s)\geq\alpha$ and $\Psi'(s)\geq\alpha'$  be two equivalent valid linear inequalities for $R_f$.
\noindent(i) The function $\psi$ is sublinear if and only if $\psi'$ is sublinear.\\
\noindent(ii) Inequality $\Psi(s)\geq\alpha$ is dominated by a minimal valid linear inequality if and only if $\Psi'(s)\geq\alpha'$ is dominated  by a minimal valid linear inequality. In particular, $\Psi(s)\geq\alpha$ is minimal if and only if $\Psi'(s)\geq\alpha'$ is minimal.
\end{lemma}
\begin{proof}
Since $\Psi(s)\geq \alpha$ and $\Psi'(s)\geq \alpha'$ are equivalent, by definition there exist $\rho>0$ and
$\lambda\in\R^\ell$, such that $\psi(r)=\rho\psi'(r)+\lambda^T Cr$ and
$\alpha=\rho\alpha'+\lambda^T(d-Cf)$. This proves (i).

Point (ii) follows from the fact that, given a function $\bar\psi'$ such that $\bar\psi'(r)\leq \psi'(r)$ for every $r\in W$, then the function $\bar\psi$ defined by $\bar\psi(r)=\rho\bar\psi'(r)+\lambda^T Cr$, $r\in W$, satisfies $\bar\psi(r)\leq\psi(r)$ for every $r\in W$. Furthermore $\bar\psi(r)<\psi(r)$ if and only if $\bar\psi'(r)< \psi'(r)$.
\end{proof}
\medskip

Given a nontrivial valid linear inequality $\Psi(s)\geq
\alpha$ for $R_f$ such that $\psi$ is sublinear, we consider the set
$$B_{\psi}=\{x\in f+W\st \psi(x-f)\leq \alpha\}.$$

Since $\psi$ is continuous, $B_\psi$ is closed. Since
$\psi$ is convex, $B_\psi$ is convex. Since $\psi$ defines a valid inequality, $B_\psi$ is lattice-free. Indeed the interior of $B_\psi$ is $\intr(B_{\psi})=\{x\in f+W\,:\, \psi(x-f)< \alpha\}$. Its boundary is $\bd(B_{\psi})=\{x\in f+W\,:\, \psi(x-f)= \alpha\}$, and its recession cone is $\rec(B_{\psi})=\{x\in f+W\,:\, \psi(x-f)\leq 0\}$. Note that $f$ is in the interior of $B_\psi$ if and only if $\alpha>0$ and $f$ is on the boundary if and only if $\alpha=0$.

\begin{remark}\label{rmk:gauge} Given a linear inequality of the form $\Psi(s)\geq 1$ such that $\psi(r)\geq 0$ for all $r\in W$,
$$\psi(r)=\inf\{t>0\st f+t^{-1}r\in B_\psi\},\quad r\in W.
$$
\end{remark}
\begin{proof} Let $r\in W$. If $\psi(r)>0$, let $t$ be the minimum positive number such that $f+t^{-1} r\in B_\psi$. Then $f+t^{-1}r\in\bd(B_\psi)$, hence $\psi(t^{-1} r)=1$ and by positive homogeneity $\psi(r)=t$. If $\psi(r)=0$, then $r\in\rec(B_\psi)$, hence $f+t^{-1}r\in B_\psi$ for every $t>0$, thus the infimum in the above equation is $0$.\end{proof}

This remark shows that, if $\psi$ is nonnegative, then it is the {\em gauge} of the convex set $B_\psi-f$ (see~\cite{lemar}).
\bigskip

Before proving Theorem~\ref{thm:min-ineq-intr}, we need the following general theorem about sublinear functions. Let $K$ be a closed, convex set in $W$ with the origin in its interior. The {\em polar} of $K$ is the set $K^*=\{y\in W\st ry\leq 1 \mbox{ for all } r\in K\}$. Clearly $K^*$ is closed and convex, and since $0\in \intr(K)$, it is well known that $K^*$ is bounded. In particular, $K^*$ is a compact set. Also, since $0\in K$, $K^{**}=K$ (see~\cite{lemar} for example). Let\begin{equation}\label{eq:set-T} \hat K=\{y\in K^*\st \exists x\in K\mbox{ such that } xy=1\}.\end{equation} Note that $\hat K$ is contained in the relative boundary of $K^*$. Let $\rho_K:\,W\rightarrow \R$ be defined by
\begin{equation}\label{eq:support} \rho_K(r)=\sup_{y\in \hat K}ry,\quad \mbox{for all } r\in W.\end{equation}
It is easy to show that $\rho_K$ is sublinear.

\begin{theorem}[Basu et al.~\cite{BaCoZa}]\label{thm:sublinear}
Let $K\subset W$ be a closed convex set containing the origin in its interior. Then $K=\{r\in W\st \rho_K(r)\leq 1\}$. Furthermore, for every sublinear function $\sigma$ such that $K=\{r\st \sigma(r)\leq 1\}$, we have $\rho_K(r)\leq \sigma(r)$ for every $r\in W$.
\end{theorem}

\begin{remark}\label{rmk:K-poly} Let $K\subset W$ be a polyhedron containing the origin in its interior. Let $a_1,\ldots, a_t\in W$ such that $K=\{r\in W\st a_ir\leq 1,\, i=1,\ldots, t\}$. Then $\rho_K(r)=\max_{i=1,\ldots,t}a_i r$.
\end{remark}
\begin{proof} The polar of $K$ is $K^*=\conv\{0,a_1,\ldots,a_t\}$ (see Theorem 9.1 in Schrijver~\cite{sch}). Furthermore,  $\hat K$ is the union of all the facets of $K^*$ that do not contain the origin, therefore
$$\rho_{K}(r)=\sup_{y\in\hat K}yr=\max_{i=1,\ldots,t}a_ir$$ for all $r\in W$.
\end{proof}

\begin{remark}\label{rmk:valid} Let $B$ be a closed lattice-free convex set in $f+W$ with $f$ in its interior, and let $K=B-f$. Then the inequality $\sum_{r\in W}\rho_K(r)s_r\geq 1$ is valid for $R_f$.
\end{remark}
\begin{pf} Let $s\in R_f$. Then $x=f+\sum_{r\in W}rs_r$ is integral, therefore $x\notin \intr(B)$ because $B$ is lattice-free. By Theorem~\ref{thm:sublinear}, $\rho_K(x-f)\geq 1$. Thus
$$1\leq \rho_K(\sum_{r\in W}rs_r)\leq \sum_{r\in W}\rho_K(rs_r)\leq \sum_{r\in W}\rho_K(r)s_r,$$
where the second inequality follows from the subadditivity of $\rho_K$ and the last from the positive homogeneity.
\end{pf}

\begin{lemma}\label{rem:psi_rho}
Given a maximal lattice-free convex set $B$ of $f+W$  containing $f$ in its interior,  $\Psi_B(s)\geq 1$ is a minimal valid inequality for $R_f$.
\end{lemma}
\begin{proof} Let $\Psi(s)\geq 1$ be a valid linear inequality for $R_f$ such that $\psi(r)\leq \psi_B(r)$ for all $r\in W$. Then $B_\psi\supset B$ and $B_\psi$ is lattice-free. By maximality of $B$, $B=B_\psi$. By Theorem~\ref{thm:sublinear} and Remark~\ref{rmk:K-poly}, $\psi_B(r)\leq\psi(r)$ for all $r\in W$, proving $\psi=\psi_B$.
\end{proof}

\vspace{.1in}

%\noindent{\em Proof of Theorem~\ref{thm:min-ineq-intr}.}\bigskip
\begin{proof}[Proof of Theorem~\ref{thm:min-ineq-intr}]$\,$

\noindent Let $\Psi(s)\geq \alpha$ be a nontrivial valid linear inequality for $R_f$. By Lemma~\ref{lemma:sublinear}, we may assume that $\psi$ is sublinear.

\begin{claim}\label{claim:BcapV} If $\intr(B_\psi)\cap V=\emptyset$, then $\Psi(s)\geq \alpha$ is trivial.
\end{claim}
Suppose $\intr(B_\psi)\cap V=\emptyset$ and let $s\in\V$ such that $s_r \geq 0$ for every $r \in W$. Let $x=f+\sum_{r\in W}r s_r$. Since $s\in\V$, $x\in V$, so $x\notin\intr(B_\psi)$. This implies $$\alpha\leq \psi(x-f)=\psi(\sum_{r\in W}rs_r)\leq \sum_{r\in W}\psi(r)s_r=\Psi(s),$$ where the last inequality follows from the sublinearity of $\psi$. \hfill $\diamond$\medskip

\begin{claim}\label{claim:finV} If $f\in V$ and $\alpha\leq 0$, then $\intr(B_\psi)\cap V=\emptyset$.
\end{claim}
Suppose $f\in V$, $\alpha \leq 0$ but $\intr(B_\psi)\cap V\neq\emptyset$. Then $\dim(\intr(B_\psi)\cap V)=\dim(V)$, hence $\intr(B_\psi)\cap V$ contains a set $X$ of $\dim(V)+1$ affinely independent points. For every $x\in X$ and every $\lambda>0$, $\psi(\lambda (x-f))=\lambda\psi(x-f)<0$, where the last inequality is because $x\in\intr(B_\psi)$. Hence the set $\Gamma=f+\cone\{x-f\st x\in X\}$ is contained in $\intr(B_\psi)$. Since $\Gamma$ has dimension equal to $\dim(V)$ and $V$ is the convex hull of its integral points, $\Gamma\cap \Z^q\neq 0$, contradicting the fact that $B_{\psi}$ has no integral point in its interior.\hfill $\diamond$\medskip

\begin{claim}\label{claim:fnotinV} If $f\notin V$, then there exists a valid linear inequality  $\Psi'(s)\geq 1$ for $R_f$ equivalent to $\Psi(s)\geq \alpha$.
\end{claim}
Since $f\notin V$, $Cf\neq d$, hence there exists a row $c_i$ of $C$ such that $d_i-c_i f\neq 0$. Let $\lambda=(1-\alpha)(d_i-c_i f)^{-1}$, and define $\psi'(r)=\psi(r)+\lambda c_i r$ for every $r\in W$. The inequality $\Psi'(s)\geq 1$ is equivalent to $\Psi(s)\geq\alpha$.\hfill $\diamond$\medskip

Thus, by Claims~\ref{claim:BcapV},~\ref{claim:finV} and~\ref{claim:fnotinV} there exists a valid linear inequality  $\Psi'(s)\geq 1$ for $R_f$ equivalent to $\Psi(s)\geq \alpha$. By Lemma~\ref{lemma:equiv}, $\psi'$ is sublinear and $\Psi(s)\geq\alpha$ is dominated by a minimal valid linear inequality if and only if $\Psi'(s)\geq\alpha'$ is dominated  by a minimal valid linear inequality. Therefore we only need to consider valid linear inequalities of the form $\Psi(s)\geq 1$ where $\psi$ is sublinear.  In particular the set $B_\psi=\{x\in W\st \psi(x-f)\leq 1\}$ contains $f$ in its interior.
\bigskip

Let $K=\{r\in W\st \psi(r)\leq 1\}$, and let $\hat K$ be defined as in~(\ref{eq:set-T}).
\begin{claim}\label{claim:sigma} The inequality $\sum_{r\in W}\rho_K(r)s_r\geq 1$ is valid for $R_f$ and $\psi(r)\geq \rho_K(r)$ for all $r\in W$.
\end{claim}
Note that $B_\psi=f+K$.  Thus, by Remark~\ref{rmk:valid}, $\sum_{r\in W}\rho_K(r)s_r\geq 1$ is valid for $R_f$. Since $\psi$ is sublinear, it follows from Theorem~\ref{thm:sublinear} that $\rho_K(r)\leq \psi(r)$ for every $r\in W$. \hfill $\diamond$\medskip

By Claim~\ref{claim:sigma}, since $\rho_K$ is sublinear, we may assume that $\psi=\rho_K$.

\begin{claim}\label{claim:Bpoly} There exists a valid linear inequality $\Psi'(s)\geq 1$ for $R_f$ dominating $\Psi(s)\geq 1$ such that $\psi' $ is sublinear, $B_{\psi'}$ is a polyhedron, and $\rec(B_{\psi'}\cap V)=\lin(B_{\psi'}\cap V)$.
\end{claim}
Since $B_\psi$ is a lattice-free convex set, it is contained in some maximal lattice-free convex set $S$ by Corollary~\ref{cor:Bexists}. The set $S$ satisfies one of the statements (i)-(iii) of Theorem~\ref{thm:main}. By Claim~\ref{claim:BcapV}, $\intr(S)\cap V\neq \emptyset$, hence case (iii) does not apply. Case (ii) does not apply because $\dim(S)=\dim(B_\psi)=\dim(W)$. Therefore case (i) applies. Thus $S$ is a polyhedron and $S\cap V$ is a maximal lattice-free convex set in $V$. In particular, by Theorem~\ref{thm:lattice-free}, $\rec(S\cap V)=\lin(S\cap V)$. Since $S$ is a polyhedron containing $f$ in its interior, there exists $A\in\R^{t\times q}$ and $b\in\R^t$ such that $b_i>0$, $i=1,\ldots,t$, and $S=\{x\in f+W\st A(x-f)\leq b\}$. Without loss of generality, we may assume that $\sup_{x\in B_\psi} a_i (x-f)=1$ where $a_i$ denotes the $i$th row of $A$, $i=1,\ldots, t$. By our assumption, $\sup_{r\in K} a_i r=1$. Therefore  $a_i\in K^*$, since $a_i r\leq 1$ for all $r\in K$. Furthermore $a_i\in\cl(\hat K)$, since $\sup_{r\in K} a_i r=1$.

Let $\bar S=\{x\in f+W\st A(x-f)\leq e\}$, where $e$ denotes the vector of all ones. Then $B_\psi\subseteq \bar S\subseteq S$.
Let $Q=\{r\in W\st Ar\leq e\}$. By Remark~\ref{rmk:K-poly},
$\rho_{Q}(r)=\max_{i=1,\ldots,t}a_ir$ for all $r\in W$. Since $\bar S\subseteq S$, $\bar S$ is lattice-free, by Remark~\ref{rmk:valid} the inequality $\sum_{r\in W} \rho_{Q}(r)s_r\geq 1$ is valid for $R_f$. Furthermore, since $\{a_1,\ldots,a_t\}\subset \cl(\hat K)$, by Claim~\ref{claim:sigma} we have
$$\psi(r)=\sup_{y\in \hat K}y r\geq \max_{i=1,\ldots,t}a_i r=\rho_{Q}(r)$$
for all $r\in W$.  Let $\psi'=\rho(Q)$. Note that $B_{\psi'}=\bar S$. So, $\rec(B_{\psi'})=\rec(\bar S) =
\{r\in W\st Ar\leq 0\}=\rec(S)$. Since $\rec(S\cap V)=\lin(S\cap
V)$, then $\rec(B_{\psi'}\cap V)=\lin(B_{\psi'}\cap V)$.\hfill $\diamond$ \medskip

By Claim~\ref{claim:Bpoly}, we may assume that $B_\psi=\{x\in f+W\st A(x-f)\leq e\}$, where $A\in \R^{t\times q}$ and $e$ is the vector of all ones, and that $\rec(B_{\psi}\cap V)=\lin(B_{\psi}\cap V)$. Let $a_1,\ldots, a_t$ denote the rows of $A$. By Claim~\ref{claim:sigma} and Remark~\ref{rmk:K-poly}, \begin{equation}\label{eq:psi}\psi(r)=\max_{i=1,\ldots,t}a_i r,\quad \mbox{for all } r\in W.\end{equation}
Let $G$ be a matrix such that $W=\{r\in\R^q\st Gr=0\}$.

\begin{claim}\label{claim:lambda} There exists $\lambda\in\R^\ell$ such that $\psi(r)+\lambda^T C r\geq 0$ for all $r\in W$.
\end{claim}
Given $\lambda\in\R^\ell$, then by (\ref{eq:psi}) $\psi(r)+\lambda^T
C r\geq 0$ for every $r\in W$ if and only if $\min_{r\in W}
(\max_{i=1,\ldots,t}a_i r + \lambda^TCr)=0$.  The latter holds if
and only if
$$0=\min\{z+\lambda^TCr\st ez-Ar\geq 0,\, Gr=0\}.$$ By LP duality,
this holds if and only if the following system is feasible
\begin{eqnarray*}
ey&=&1\\
A^Ty+C^T\lambda-G^T\mu&=&0\\
y&\geq& 0.
\end{eqnarray*}
Clearly the latter is equivalent to
\begin{eqnarray}\label{eq:lambda}
A^Ty+C^T\lambda-G^T\mu &=&0\\
y\geq 0,\,
y \neq 0.&&\nonumber
\end{eqnarray}
Note that $\rec(B_\psi\cap V)=\{r\in \R^q\st Ar\leq 0,\, Cr=0, Gr=0\}$ and $\lin(B_\psi\cap V)=\{r\in \R^q\st Ar= 0,\, Cr=0, Gr=0\}$.
Since $\rec(B_\psi\cap V)=\lin(B_\psi\cap V)$, the system
\begin{eqnarray*}
Ar&\leq& 0\\
Cr&=&0\\
Gr&=&0\\
e^TAr&=&-1
\end{eqnarray*}
is infeasible. By Farkas Lemma, this is the case if and only if there exists $\gamma\geq 0$, $\lambda$, $\tilde\mu$, and $\tau$ such that
$$
A^T\gamma+C^T\lambda+G^T\tilde\mu+A^Te\tau=0,\quad \tau>0.
$$
If we let $y=\gamma+e\tau$ and $\mu =-\tilde \mu$, then $(y,\lambda,\mu)$ satisfies~(\ref{eq:lambda}). By the previous argument, $\lambda$ satisfies the statement of the claim.\hfill $\diamond$\medskip

Let $\lambda$ as in Claim~\ref{claim:lambda}, and let $\psi'$ be the function defined by $\psi'(r)=\psi(r)+\lambda^TCr$ for all $r\in W$. So $\psi'(r)\geq 0$ for every $r\in W$. Let $\alpha'=1+\lambda^T(d-Cf)$. Then the inequality $\Psi'(s)\geq \alpha'$ is valid for $R_f$ and it is equivalent to $\Psi(s)\geq \alpha$. If $\alpha'\leq 0$, then $\Psi'(s)\geq \alpha'$ is trivial. Thus  $\alpha'>0$. Let $\rho=1/{\alpha'}$ and let $\psi''=\rho\psi'$. Then $\Psi''(s)\geq 1$ is equivalent to $\Psi(s)\geq 1$. By Lemma~\ref{lemma:equiv}(i), $\psi''$ is sublinear.

Let $B$ be a maximal lattice-free convex set of $f+W$ containing $B_{\psi''}$. Such a set $B$ exists by Corollary~\ref{cor:Bexists}.

\begin{claim}\label{claim:psiB} $\psi''(r)\geq\psi_B(r)$ for all $r\in W$.\end{claim}

Let $r\in \rec(B_{\psi''})$. Since $\psi''$ is nonnegative,  $\psi''(r)=0$. Since $\rec(B_{\psi''})\subseteq\rec(B)$, $\psi_{B} (r)\leq 0=\psi''(r)$. Let $r\notin \rec(B_{\psi''})$. Then $f+\tau r\in\bd(B_{\psi''})$ for some $\tau>0$, hence $\psi''(\tau r)=1$ and, by positive homogeneity, $\psi''(r)=\tau^{-1}$. Because $B_{\psi''}\subset B$, $f+\tau r\in B$. Since $B=\{x\in f+W\st \psi_B(x-f)\leq 1\}$, it follows that $\psi_B(\tau r)\leq 1$, implying $\psi_{B}(r)\leq \tau^{-1}=\psi''(r)$. \hfill $\diamond$ \medskip
\medskip

Claim~\ref{claim:psiB} shows that $\Psi''(s)\geq 1$ is dominated by $\Psi_B(s)\geq 1$, which is minimal by Lemma~\ref{rem:psi_rho}. By Lemma~\ref{lemma:equiv}(ii), $\Psi(s)\geq 1$ is dominated by a minimal valid linear inequality which is equivalent to $\Psi_B(s)\geq 1$.
\end{proof}
%\hfill $\Box$

\vspace{.1in}

\noindent{\em Example. }
We illustrate the end of the proof in an example. Suppose $W=\{x\in\R^3\st x_2+\sqrt{2}  x_3=0\}$, and let $f=(\frac 12, 0, 0)$. Note that $f+W=W$. All integral points in $W$ are of the form $(k,0,0)$, $k\in\Z$, hence $V=\{x\in W\st x_2=0\}$. Thus $\V=\{s\in \W\st \sum_{r\in W}r_2s_r=0\}$.

Consider the function $\psi\,:\,W\rightarrow \R$ defined by
$\psi(r)=\max\{-4r_1-4 r_2,4r_1-4r_2\}$. The set $B_\psi=\{x\in W\st
-4(x_1-\frac 1 2)-4 x_2\leq 1,\, 4(x_1-\frac 1 2)-4 x_2\leq 1\}$
does not contain any integral point, hence  $\Psi(s)\geq 1$ is valid
for $R_f$. Note that $B_\psi$ is not maximal (see
Figure~\ref{fig:example}).

Setting $\lambda=4$ in Claim~\ref{claim:lambda}, let
$\psi'(r)=\psi(r)+\lambda r_2$ for all $r\in W$. Note that
$\psi'(r)=\max\{-4r_1,4r_1\}\geq 0$ for all $r\in W$. The set
$B_{\psi'}=\{x\in W\st -4(x_1-\frac 1 2)\leq 1,\, 4(x_1-\frac 1
2)\leq 1\}$ is contained in the maximal lattice-free convex set
$B=\{x\in W\st -2(x_1-\frac 1 2)\leq 1,\, 2(x_1-\frac 1 2)\leq 1\}$,
hence $\psi'$ is pointwise larger than the function $\psi_B$ defined
by $\psi_B(r)=\max\{-2r_1,2r_1\}$ and $\Psi_B(s)\geq 1$ is valid for
$R_f$. This completes the illustration of the proof.

\begin{figure}[htbp]\begin{center}
\includegraphics[scale=0.65]{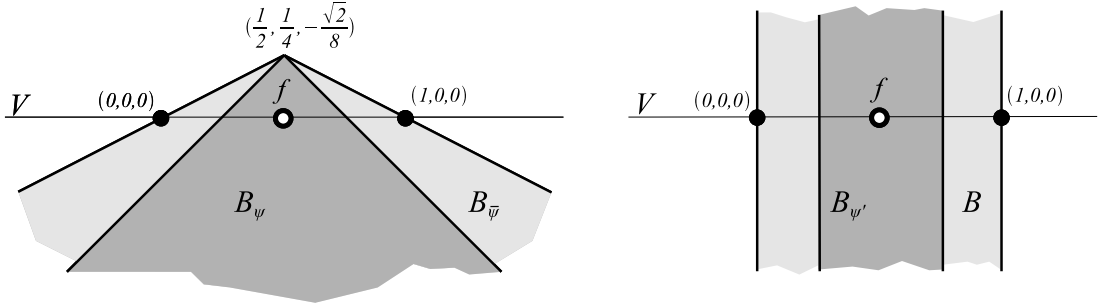} %\includegraphics[scale=0.65]{example1.eps}
\caption{\label{fig:example} Lattice-free sets in the 2-dimensional
space $W$.}
\end{center}
\end{figure}

Note that $\Psi_B(s) \geq 1$ does not dominate $\Psi(s) \geq 1$.
However the inequality $\Psi_B(s)\geq 1$ is equivalent to a valid
inequality $\overline\Psi(s) \geq 1$ which dominates $\Psi(s) \geq
1$. We show how to construct $\bar\psi$ in our example. The function
$\bar \psi$ is defined by $\psi_B(r)-\lambda r_2$ for all $r\in W$
is pointwise smaller than $\psi$ and $\bar\Psi(s)\geq 1$ is valid
for $R_f$. Moreover, $B_{\bar\psi}=\{x\in W\st -2(x_1-\frac 1
2)-4x_1\leq 1,\, 2(x_1-\frac 1 2)-4x_1\leq 1\}$ is a maximal
lattice-free convex set containing $B_\psi$. Note that the recession
cones of $B_{\psi}$ and $ B_{\bar\psi}$ are full dimensional, hence
$\psi$ and $\bar \psi$ take negative values on elements of the
recession cone. For example $\psi(0,-1, \frac 1 {\sqrt{2}})=\bar
\psi(0,-1, \frac 1 {\sqrt{2}})=-4$. The recession cones of
$B_{\psi'}$ and $B$ coincide and are not full dimensional, thus
$\psi'(0,-1, \frac 1 {\sqrt{2}})=\psi_B(0,-1, \frac 1
{\sqrt{2}})=0$, since the vector $(0,-1, \frac 1 {\sqrt{2}})$ is in
the recession cone of $B$.

\section{The intersection of all minimal inequalities}\label{sec:closure}

In this section we prove Theorem~\ref{thm:closure}. First we need the following.

\begin{lemma}\label{lemma:Psi}
Let $\psi\,:\,W\rightarrow \R$ be a continuous function that is
positively homogeneous. Then the function $\Psi\,:\,\W\rightarrow \R$, defined by $\Psi(s)=\sum_{r\in W}\psi(r)s_r$, is continuous with respect to
$(\W,\|\cdot\|_H)$.
\end{lemma}
\begin{pf}
Define $\gamma=\sup\{|\psi(r)|\,:\, r\in W,\,\|r\|=1\}$. Since the
set $\{r\in R_f(W)\,:\,\|r\|=1\}$ is compact and $\psi$ is
continuous, $\gamma$ is well defined (that is, it is finite). Given
$s,s'\in\W$, we will show $|\Psi(s')-\Psi(s)|\leq \gamma
\|s'-s\|_H$, which implies that $\Psi$ is continuous. Indeed
\begin{eqnarray*}
|\Psi(s')-\Psi(s)|&=&|\sum_{r\in W} \psi(r) (s'_r-s_r)|\\
&\leq & \sum_{r\in W} |\psi(r)|\, |s'_r-s_r|\\
&=& \sum_{r\in W\,:\,\|r\|=1}\sum_{\alpha>0}|\psi(\alpha r)| \, |s'_{\alpha r}-s_{\alpha r}|\\
&=& \sum_{r\in W\,:\,\|r\|=1}|\psi(r)|\sum_{\alpha>0}\alpha |s'_{\alpha r}-s_{\alpha r}|\quad\mbox{(by positive homogeneity of }\psi)\\
&\leq& \gamma \sum_{r\in W\,:\,\|r\|=1}\sum_{\alpha>0}\alpha |s'_{\alpha r}-s_{\alpha r}|\\
&=& \gamma \sum_{r\in W\backslash \{0\}} \|r\| |s'_r-s_r|\quad\leq \,\gamma \|s'-s\|_H
\end{eqnarray*}
\end{pf}

\begin{proof}[Proof of Theorem~\ref{thm:closure}]
``$\subseteq$'' By Lemma~\ref{lemma:Psi}, $\Psi_B$ is continuous in
$(\W,\|\cdot\|_H)$ for every $B\in\mathcal{B}_W$,
therefore $\{s\in\W\,:\, \Psi_B(s)\geq 1\}$ is a closed
half-space of  $(\W,\|\cdot\|_H)$. It is immediate to
show that also $\{s\in\W\,:\, s_r\geq 0,\, r\in W\}$ is a
closed set in $(\W,\|\cdot\|_H)$. Since $\V=\{s\in\W\st \sum_{r\in W}(Cr)s_r=d-Cf\}$, and since for each row $c_i$ of $C$ the function $r\mapsto c_i r$ is positive homogeneous, then by Lemma~\ref{lemma:Psi} $\V$ is also closed.
Thus
$$\{s\in\V\,:\, \Psi_B(s)\ge 1 , B\in\mathcal{B}_W;\;
s_r\geq 0,\,  r\in W\}$$ is an intersection of closed sets, and is
therefore a closed set of  $(\W,\|\cdot\|_H)$. Thus,
since it contains $\conv(R_f(W))$, it also contains
$\ol\conv(R_f(W))$.
\medskip

\noindent ``$\supseteq$'' We only need to show that, for every
$\bar{s}\in\V$ such that $\bar s\notin\ol\conv(R_f(W))$
and $\bar s_r\geq 0$ for every $r\in W$, there exists $B\in\mathcal{B}_W$
such that $\sum_{r\in W}\psi_B(r)\bar s_r< 1$.

The theorem of Hahn-Banach implies the following. \begin{quote}\sl
Given a closed convex set $A$ in $(\W,\|\cdot\|_H)$ and a
point $b\notin A$, there exists a continuous linear function
$\Psi\,:\,\W\rightarrow \R$ that strictly separates $A$
and $b$, i.e. for some $\alpha\in \R$, $\Psi(a)\geq \alpha$ for
every $a\in A$, and $\Psi(b)< \alpha$.\end{quote}

Therefore, there exists a linear function
$\Psi\,:\,\W\rightarrow \R$ such that $\Psi(\bar
s)<\alpha$ and $\Psi(s)\geq \alpha$ for every $s\in\ol\conv(R_f)$.
By the first part of Theorem~\ref{thm:min-ineq-intr}, we may assume that $\Psi(s)\geq \alpha$ is a nontrivial minimal
valid linear inequality. By the second part of Theorem~\ref{thm:min-ineq-intr}, this
inequality is equivalent to an inequality of the form $\sum_{r\in
W}\psi_B(r)s_r\ge 1$ for some maximal lattice-free convex set $B$ of
$W$ with $f$ in its interior.
\end{proof}

\section{Proof of Theorem~\ref{thm:extreme}}\label{sec:extreme}

By Theorem~\ref{thm:lattice-free}, $L:=\lin(B)$ is a rational space
and $P$ is a polytope. Also, by construction, $\psi_B(r^j)=1$ for
$j=1,\ldots,k$, and $\psi_B(r^j)=0$ for $j=k+1,\ldots,k+h$. Thus
$\sum_{j=1}^k s_{j}\geq 1$ is a valid inequality for
$R_f(r^1,\ldots,r^{k+h})$. We recall that,  by
Theorem~\ref{thm:min-ineq-intr}, $\sum_{r\in\R^q}\psi_B(r)s_r\geq 1$
is a minimal valid inequality for $R_f(\R^q)$.\medskip

We first show that $R_f(r^1,\ldots,r^{k+h})$ is nonempty. Since
$r^{k+1},\ldots,r^{k+h}$ are rational, there exists a positive
integer $N$ such that $Nr^j$ is integral for $j=k+1,\ldots,k+h$.
Since $f\in\intr(B)$, it follows that, for every $r\in\R^q$, there
exists $s\in\R^{k+h}$ such that $r=\sum_{j=1}^{k+h}r^js_j$ and
$s_j\geq 0$, $j=1,\ldots,k$. Thus, given $\bar x\in\Z^n$, there
exists $\bar s$ such that $\bar x-f=\sum_{j=1}^{k+h}r^j\bar s_j$ and
$s_j\geq 0$, $j=1,\ldots,k$. Let $\lambda$ be a positive integer
such that $\bar s+\lambda N\sum_{j=k+1}^{k+h}e^j\geq 0$, where $e_j$
denotes the $j$th unit vector in $\R^{k+h}$. Then $\bar s+\lambda
N\sum_{j=k+1}^{k+h}e^j\geq 0\in R_f(r^1,\ldots,r^{k+h})$.
\bigskip

``$\Leftarrow$'' Let us assume that $\sum_{j=1}^k s_{j}\geq 1$ is an
extreme inequality for $R_f(r^1,\ldots,r^{k+h})$. We show that
$\sum_{r\in\R^q}\psi_B(r)s_r\geq 1$ is extreme for $R_f(\R^q)$.

Let $\sum_{r\in \R^q}\psi_i(r)s_r\geq 1$, $i=1,2$, be valid
inequalities for $R_f(\R^q)$ such $\psi_B \geq \frac 12
(\psi_1+\psi_2)$. We will show that $\psi_B=\psi_1=\psi_2$. Since
$\psi_B$ is minimal and $\psi_B \geq \frac 12 (\psi_1+\psi_2)$, then
$\psi_1,\,\psi_2$ are both minimal. Thus, given $B_i=\{x\in
\R^q\,:\, \psi_i(x-f)\leq 1\}$, $i=1,2$, $B_1$ and $B_2$ are maximal
lattice-free convex sets. Furthermore, since $\sum_{j=1}^k s_j\geq
1$ is extreme for $R_f(r^1,\ldots,r^{k+h})$, then
$\psi_1(r^j)=\psi_2(r^j)=\psi_B(r^j)$ for $j=1,\ldots,k+h$. In
particular, $\psi_1(r^j)=\psi_2(r^j)=1$ for $j=1,\ldots,k$ and
$\psi_1(r^j)=\psi_2(r^j)=0$ for $j=k+1,\ldots,k+h$. Hence $B_1$ and
$B_2$ contain $B$, since they contain the vertices of $P$ and their
lineality space contains $r^{k+1},\ldots,r^{k+h}$. By the maximality
of $B$, $B_1=B_2=B$, therefore $\psi_B=\psi_1=\psi_2$, proving that
$\sum_{r\in \R^q}\psi_B(r)s_r\geq 1$ is extreme.
\medskip

``$\Rightarrow$'' Let us assume that  $\sum_{r\in
\R^q}\psi_B(r)s_r\geq 1$ is an extreme inequality for $R_f(\R^q)$.
We prove that $\sum_{j=1}^k s_{j}\geq 1$ is an extreme inequality
for $R_f(r^1,\ldots,r^{k+h})$.

Let $\alpha,\, \beta\in\R^{k+h}$ be vectors such that $\alpha s\geq
1$ and $\beta s\geq 1$ are valid for $R_f(r^1,\ldots,r^{k+h})$, and
$\frac 1 2 (\alpha_j+\beta_j) \leq \psi_B(r^j)$, $j=1,\ldots,k+h$.
We will show that it must follow that $\alpha_j=\beta_j=\psi_B(r^j)$
for $j=1,\ldots,k+h$. \medskip

We first observe that, for $j=k+1,\ldots,k+h$, $\alpha_j=\beta_j=0$.
If not, since $\frac 12 (\alpha_j+\beta_j)\leq \psi_B(r^j)=0$ for
$j=k+1,\ldots,k+h$, then we may assume that $\alpha_j<0$ for some
$j\in\{k+1,\ldots,k+h\}$. Given $\bar s\in R_f(r^1,\ldots,r^{k+h})$
it now follows that $\bar s+\lambda N e^j\in R_f(r^1,\ldots,r^{k+h})$ for every positive integer $\lambda$.
However, $\lim_{\lambda\rightarrow +\infty}\alpha (\bar s+\lambda N
e^j)=\alpha\bar s+\lim_{\lambda\rightarrow +\infty}\lambda N
\alpha_j=-\infty$, contradicting the fact that $\alpha s\geq 1$ is
valid for $R_f(r^1,\ldots,r^{k+h})$.

Define, for every $r\in\R^q$,
\begin{equation}\label{eq:psi-alpha}\psi^{\alpha}(r)=\min \{\alpha
s\st\sum_{j=1}^{k+h}r^j s_j=r,\,s_j\geq 0, \,
j=1,\ldots,k\}.\end{equation}

Note that, for every $r\in\R^q$, the above linear program is
feasible. We also observe that, for every $\bar x\in\Z^q$,
$\psi^\alpha(\bar x-f)\geq 1$. Indeed, given $\bar s\in\R^{k+h}$
such that $\psi^\alpha(\bar x-f)=\alpha \bar s$ and
$\sum_{j=1}^{k+h}r^j \bar s_j=\bar x-f$, $s_j\geq 0$ for
$j=1,\ldots,k$, then there exists a positive integer $\lambda$ such
that $\tilde s=\bar s+\lambda N\sum_{j=k+1}^{k+h} e^j\in
R_f(r^1,\ldots,r^{k+h})$. Since $\alpha_j=0$, $j=k+1,\ldots,k+h$,
then $\psi^\alpha(\bar x-f)=\alpha \bar s=\alpha \tilde s\geq 1$.

The above fact also implies that the linear
program~\eqref{eq:psi-alpha} admits a finite optimum for every
$r\in\R^q$.

We show that $\sum_{r\in\R^q}\psi^\alpha(r)s_r\geq 1$ is a valid
inequality for $R_f(\R^q)$. The function $\psi^\alpha$ is sublinear
(the proof is similar to the one of Lemma~\ref{lemma:sublinear}).
Therefore, for every $s\in R_f(\R^q)$, given $\bar
x=f+\sum_{r\in\R^q}r s_r$, it follows that
$$\sum_{r\in\R^q}\psi^\alpha(r)s_r\geq \psi^\alpha(\bar x-f)\geq 1.$$

We may define $\psi^\beta$ similarly. It now follows that the sets
$B_{\psi^{\alpha}}$ and $B_{\psi^{\beta}}$ are lattice-free convex
sets with $f$ in their interior.

By definition, $\psi^\alpha(r^j)\leq \alpha_j$ and
$\psi^\beta(r^j)\leq \beta_j$, $j=1,\ldots,k+h$.

Let $\psi=\frac 12 (\psi^\alpha+\psi^\beta)$. We will show that
$\psi= \psi_B$. Indeed, $\psi(r^j)\leq\frac 12
(\alpha_j+\beta_j)\leq \psi_B(r^j)$, $j=1,\ldots,k+h$. Thus
$\psi(r^j)\leq 1$ for $j=1,\ldots,k$ and $\psi(r^j)\leq 0$ for
$j=k+1,\ldots,k+h$. In particular, $f+r^1,\ldots,f+r^k\in B_\psi$
and $r^{k+1},\ldots, r^{k+h}\in\rec(B_\psi)$. Thus $B_\psi\supseteq
B$. Since $\psi$ is a convex combination of $\psi^\alpha$ and
$\psi^\beta$, it follows that $\sum_{r\in\R^q}\psi(r)s_r\geq 1$ is a
valid inequality for $R_f(\R^q)$. Thus $B_\psi$ is a lattice-free
convex set. Since $B$ is maximal, it follows that $B_\psi=B$. Hence
$\psi=\psi_B$.

Since $\psi_B$ is extreme, it follows that
$\psi^\alpha=\psi^\beta=\psi_B$. Hence
$\alpha_j=\beta_j=\psi_B(r^j)$, $j=1,\ldots,k+h$. \hfill $\Box$

\end{document}